\documentclass[11pt]{amsart}

\usepackage{amsmath,amssymb}
\usepackage{graphicx}

\usepackage{fullpage}
\usepackage{enumerate}
\usepackage{pgf,tikz}
\usetikzlibrary{arrows}
\usetikzlibrary{positioning}

\def\COMMENT#1{}
\let\COMMENT=\footnote
\def\TASK#1{}

\def\noproof{{\unskip\nobreak\hfill\penalty50\hskip2em\hbox{}\nobreak\hfill%
        $\square$\parfillskip=0pt\finalhyphendemerits=0\par}\goodbreak}
\def\endproof{\noproof\bigskip}
\newdimen\margin   
\def\textno#1&#2\par{%
    \margin=\hsize
    \advance\margin by -4\parindent
           \setbox1=\hbox{\sl#1}%
    \ifdim\wd1 < \margin
       $$\box1\eqno#2$$%
    \else
       \bigbreak
       \hbox to \hsize{\indent$\vcenter{\advance\hsize by -3\parindent
       \sl\noindent#1}\hfil#2$}%
       \bigbreak
    \fi}
\def\proof{\removelastskip\penalty55\medskip\noindent{\bf Proof. }}

\newtheorem{thm}{Theorem}[section]

\newtheorem{lemma}[thm]{Lemma}
\newtheorem{claim}[thm]{Claim}

\newtheorem{question}[thm]{Question}

\newtheorem{prop}[thm]{Proposition}

\newtheorem*{thm*}{Theorem}
\newtheorem*{define*}{Definition}
\newtheorem*{examp*}{Example}
\newtheorem*{lem*}{Lemma}
\newtheorem*{claim*}{Claim}
\newtheorem*{fact*}{Fact}
\newtheorem*{col*}{Corollary}
\newtheorem*{conj*}{Conjecture}

\def\eps{\varepsilon}

\theoremstyle{definition}
\newtheorem{define}[thm]{Definition}

\begin{document}

\title{Tilings in vertex ordered graphs}

\author{J\'ozsef Balogh}
\address{
Department of Mathematical Sciences, University of Illinois at Urbana-Champaign, IL, USA, and Moscow Institute of Physics and Technology, Russian Federation.
Partially supported by NSF Grant DMS-1764123 and Arnold O. Beckman Research Award (UIUC) Campus Research Board 18132 and the Langan Scholar Fund (UIUC).
}
\email{jobal@illinois.edu}

\author{Lina Li}
\address{
Department of Mathematics, University of Illinois at Urbana-Champaign, Urbana, IL, USA.
} 
\email{linali2@illinois.edu}

\author{Andrew Treglown}
\address{
School of Mathematics, University of Birmingham, Edgbaston, Birmingham, UK.  Research supported by EPSRC grant EP/V002279/1.
}
\email{a.c.treglown@bham.ac.uk}
\date{\today}
\begin{abstract}
Over recent years there has been much interest in both Tur\'an and Ramsey properties of \emph{vertex ordered graphs}. In this paper we initiate the study of embedding spanning structures into vertex ordered graphs. In particular, we introduce a general framework for approaching the problem of determining the minimum degree threshold for forcing a \emph{perfect $H$-tiling} in an ordered graph. In the (unordered) graph setting, this problem was resolved by K\"uhn and Osthus [The minimum degree threshold for perfect graph packings, Combinatorica, 2009]. We use our general framework to resolve the perfect $H$-tiling problem for all ordered graphs $H$ of interval chromatic number $2$. Already in this restricted setting the class of extremal examples is richer than in the  unordered graph problem. In the process of proving our results, novel approaches to both the regularity and absorbing methods are developed. 
\end{abstract}
\maketitle

\section{Introduction} \label{Introduction}
Over recent years there has been interest in extending classical graph theory results to the setting of \emph{vertex ordered graphs}. 
A \emph{(vertex) ordered graph} or \emph{labelled graph} $H$ on $h$ vertices is a graph whose vertices have been labelled with $[h]:=\{1,\dots,h\}$.
An ordered graph $G$ with vertex set $[n]$ \emph{contains} an ordered graph $H$ on $[h]$ if (i) there is a mapping $\phi : [h] \rightarrow [n]$ such that
$\phi (i) < \phi (j)$ for all $1\leq i < j \leq h$ and  (ii) $\phi(i)\phi(j)$ is an edge in $G$ whenever $ij$ is an edge in $H$.

A foundation stone in extremal graph theory is Tur\'an's theorem which determines the number of edges in the densest $K_r$-free graph on $n$ vertices.
Furthermore, for every graph $H$,
the Erd\H{o}s--Stone--Simonovits theorem~\cite{es1, es2} determines, up to a small error term, the number of edges in the densest $H$-free $n$-vertex graph.
It is natural to seek Tur\'an-type results in the setting of ordered graphs. Indeed, this question was first raised by F\"uredi and Hajnal~\cite{fu},
and there are now many results in the area; see  Tardos~\cite{tardosbcc} for a survey of such results (and the related problem of Tur\'an-type results for \emph{edge} ordered graphs).
In particular, Pach and Tardos~\cite{pt} proved an analogue of the  Erd\H{o}s--Stone--Simonovits theorem in the setting of ordered graphs. In their result they show that the so-called \emph{interval chromatic number}
governs the threshold (for graphs $H$ of interval chromatic number at least $3$),
rather than the chromatic number (as is the case in the unordered setting). There are  several Tur\'an-type results  for ordered graphs of interval chromatic number $2$; see e.g.~\cite{new1, gy, kor, marc, pt, tardosbcc}, as well as  Tur\'an-type results for ordered hypergraphs, see~\cite{new2}.

There have also been a number of recent results concerning  Ramsey theory for ordered graphs, for example see the work of Balko,  Cibulka,  Kr\'al and Kyn\v{c}l~\cite{balko} and of
Conlon, Fox, Lee and Sudakov~\cite{con}.

In this paper we initiate the study of embedding \emph{spanning} structures in ordered graphs. In particular, we study the minimum degree required to ensure an ordered graph has a \emph{perfect $H$-tiling}.
In both the ordered and unordered settings, 
an \emph{$H$-tiling} in a graph $G$ 
is a collection of vertex-disjoint copies of $H$ contained in $G$.
 An
$H$-tiling is \emph{perfect} if it covers all the vertices of $G$.
Perfect $H$-tilings are also often referred to as \emph{$H$-factors}, \emph{perfect $H$-packings} or \emph{perfect $H$-matchings}. 
$H$-tilings can be viewed as generalisations of both the notion of a matching (which corresponds to the case when $H$ is a single edge) and the Tur\'an problem (i.e.~a copy of $H$ in $G$ is simply an $H$-tiling of size one).

A central result in the area is the Hajnal--Szemer\'edi theorem~\cite{hs} from 1970, 
which characterises the minimum degree that ensures a graph contains a perfect \(K_r\)-tiling. 

\begin{thm}[Hajnal and Szemer\'edi~\cite{hs}]\label{hs}
Every graph $G$ whose order $n$
is divisible by \(r\) and whose minimum degree satisfies $\delta (G) \geq (1-1/r)n$ contains a perfect $K_r$-tiling. Moreover, there are $n$-vertex graphs $G$
 with $\delta (G) = (1-1/r)n-1$ that do not contain a perfect $K_r$-tiling.
\end{thm}
There has also been significant interest in the minimum degree threshold that ensures a perfect $H$-tiling for an arbitrary graph $H$.
After  earlier work on this topic (see~e.g.~\cite{alonyuster, kssAY}),  
K\"uhn and Osthus~\cite{kuhn, kuhn2}  determined, up to an additive constant, the minimum degree that forces a perfect $H$-tiling for \emph{any} fixed graph $H$.
In particular, they showed that, depending on $H$, the minimum degree threshold is governed by either the 
chromatic number $\chi (H)$ of $H$ or the
so-called \textit{critical chromatic number of $H$}.

\begin{define}[Critical chromatic number]
The \textit{critical chromatic number} $\chi_{cr}(F)$ of an unordered graph $F$ is defined as 
\[
\chi_{cr}(F):=(\chi(F) - 1)\frac{|F|}{|F| - \sigma(F)},
\]
where $\sigma(F)$ denotes the size of the smallest possible color class in any $\chi(F)$-coloring of $F$.
\end{define}

\begin{thm}[K\"uhn and Osthus~\cite{ kuhn2}]\label{kothm}
Let $\delta(H, n)$ denote the smallest integer $k$ such that every graph $G$ whose order $n$ is divisible by $|H|$ and with $\delta(G)\geq k$ contains a perfect $H$-tiling. 
For every unordered graph $H$, 
\[
\delta(H, n)=\left(1 - \frac{1}{\chi^*(H)}\right)n +O(1), 
\]
where  $\chi^*(H):=\chi_{cr} (H)$ if $\text{hcf}(H)=1$ and $\chi^*(H):=\chi (H)$ otherwise.
\end{thm}
The definition of $\text{hcf}(H)=1$ is somewhat involved; see~\cite[Section 1.2]{kuhn2} for the definition and several illuminating examples. The moral behind the 
 dichotomy in Theorem~\ref{kothm}, however, is rather straightforward to articulate. Indeed, it
 arises as there are two types of extremal construction for this problem: so-called \emph{space barriers} (which `dominate' when $\chi^*(H)=\chi_{cr} (H)$) and \emph{divisibility barriers} (which `dominate' when $\chi^*(H)=\chi (H)$).

In this paper we show that the corresponding problem for ordered graphs has a rich behaviour.
Indeed, our main result resolves the problem for all ordered graphs $H$ of interval chromatic number $2$. Even in this restricted case
the nature of the minimum degree threshold is diverse, with a range of extremal examples coming into play, including a construction which is neither a divisibility nor space barrier.~Whilst we do not resolve the problem for all ordered graphs $H$, in Section~\ref{sec:gen} we introduce a framework that can be used to attack the problem in general. Moreover, another  contribution of the paper is the approaches we develop. Indeed, as we will discuss in Section~\ref{sec:expo}  we develop an approach to applying Szemer\'edi's regularity lemma~\cite{sze} and a (\emph{local-global}) philosophy for absorbing, both of which we believe are applicable to other embedding problems for ordered graphs. In particular, 
 a key property of  the regularity method -- which is regularly used to help embed (spanning) subgraphs in the unordered setting -- breaks down for ordered graphs; we introduce an approach to overcome this. 

\subsection{Our results}
Denote by $\delta_{<}(H, n)$ the smallest integer $k$ such that every ordered graph $G$ whose order $n$ is divisible by $|H|$ and with $\delta(G)\geq k$ contains a perfect $H$-tiling. The goal of this paper is to study $\delta_{<}(H, n)$; for this we will need a few definitions. In particular, while the chromatic number is a relevant parameter in the study of perfect $H$-tilings in graphs, the 
\textit{interval chromatic number} plays a  role in the study of perfect $H$-tilings in ordered graphs.

\begin{define}[Interval chromatic number]
The \textit{interval chromatic number} $\chi_{<}(H)$ of an ordered graph $H$ is the minimum number of intervals the
vertex set $[h]$ of $H$ can be partitioned into, so that no two vertices belonging
to the same interval are adjacent in $H$.
\end{define}
As well as $\chi_{<}(H)$, another parameter $\alpha^*(H)$ plays a role in the study of perfect $H$-tilings for ordered graphs with $\chi_{<}(H)=2$. To introduce $\alpha^*(H)$ we need the following definitions.

Let $\alpha^{+}_0(H):=0$. For every $1\leq \ell\leq \chi_{<}(H)$, we let
\begin{equation}
\alpha^{+}_{\ell}(H):=\text{the largest } k\in \mathbb{N} \text{ such that } [\alpha^{+}_{\ell-1}(H) + 1,\ k] \text{ is an independent set in }H.
\end{equation}
By the definition of interval chromatic number, we always have $\alpha^{+}_{\chi_{<}(H)}(H)=h$ and therefore $\bigcup_{\ell=1}^{\chi_{<}(H)} [\alpha^{+}_{\ell-1}(H) + 1,\ \alpha^{+}_{\ell}(H)]$ is a natural partition of $[h]$ into intervals, each spanning an independent set.
We also define such parameters in the reverse order. 
Let $\alpha^{-}_0(H):=h+1$. For every $1\leq\ell\leq \chi_{<}(H)$, we let
\begin{equation*}
\alpha^{-}_{\ell}(H):=\text{ the smallest } k\in \mathbb{N} \text{ such that }[k,\ \alpha^{-}_{\ell-1}(H) - 1] \text{ is an independent set in }H.
\end{equation*}
Similarly, we have $\alpha^{-}_{\chi_{<}(H)}(H)=1$ and therefore $\bigcup_{\ell=1}^{\chi_{<}(H)}[\alpha^{-}_{\ell}(H),\ \alpha^{-}_{\ell-1}(H) - 1]$ is a natural partition of $[h]$. 
We then define 
\begin{equation}\label{def:alphastar}
\alpha^*(H):=\min_{1\leq \ell< \chi_{<}(H)} \min \left\{\frac{\alpha_{\ell}^{+}(H)}{\ell\cdot h},\ \frac{h - \alpha_{\ell}^{-}(H) +1}{\ell\cdot h}\right\}.
\end{equation}
When the underlying graph is clear, we simply write $\alpha^{+}_{\ell}$, $\alpha^{-}_{\ell}$ and $\alpha^*$.

The following proposition shows that for any ordered graph $H$, the parameter $\alpha^*(H)$ provides a lower bound for $\delta_{<}(H, n)$.
\begin{prop}\label{prop:spabarr}
Let $H$ be an ordered graph on $h$ vertices.
For every  $n\in \mathbb N$ with $h|n$, there is an $n$-vertex ordered graph $G$
 with $\delta (G) \geq \left\lfloor(1-\alpha^*(H))n\right\rfloor -1$ that does not contain a perfect $H$-tiling.
\end{prop}

The main goal of this paper is to determine the asymptotics of $\delta_{<}(H, n)$ for graphs $H$ with interval chromatic number $2$. It turns out the value of $\delta_{<}(H, n)$ in this case depends on  structural properties of $H$ encapsulated by the following three definitions.
\begin{define}[Property~A]\label{def:Pa}
An ordered graph $H$ on $h$ vertices is said to have \textit{Property~A} if $H$ has no edges in the intervals $[1, \lfloor h/2 \rfloor+1]$ and $[\lceil h/2 \rceil, h]$.
\end{define}
Note that an ordered graph $H$ has Property A if and only if $\alpha^*(H)> 1/2$.
\begin{define}[Property~B]\label{def:Pb}
An ordered graph $H$ on $h$ vertices is said to have \textit{Property~B} if for all partitions of $[ h]$ into two non-empty intervals $[1, i]$ and $[i+1, h]$, there is an edge between these two intervals.
\end{define}

Let $H$ be an ordered graph on $h$ vertices. If $h$ is not isolated then
let $s(H)$ be the smallest vertex in $H$ that is adjacent to $h$.
Similarly, if $1$ is not isolated then  let $l(H)$ be the largest vertex in $H$ that is adjacent to $1$.
 
\begin{define}[Property~C]\label{def:Pc}
For an ordered graph $H$ on $h$ vertices, the vertex $h$ is said to have \textit{Property~C} 
if $h$ is not isolated, and there exists an edge in the interval $[s(H), h-1]$.
Similarly, the vertex $1$ is said to have  \textit{Property~C}
if $1$ is not isolated and there exists an edge in the interval $[2, l(H)]$.
\end{define}

Our main theorem shows that for any ordered graph $H$ with interval chromatic number 2, either its interval chromatic number $\chi_{<}(H)$ or the new graph parameter $\alpha^*(H)$ governs the minimum degree threshold that forces the existence of a perfect $H$-tiling in ordered graphs of large minimum degree.
\begin{thm}\label{mainthm:intchro2}
Let $H$ be an ordered graph on $h$ vertices with $\chi_{<}(H)=2$.
\begin{itemize}
\item[(i)] Suppose that $H$ does not have Property~A. Then 
\[\delta_{<}(H, n)=(1 - \alpha^*(H) + o(1))n.\]
\item[(ii)]  Suppose that $H$ has both Property~A and Property~B. Then 
\[\delta_{<}(H, n)=\left (1-\frac{1}{\chi_{<}(H)}+o(1) \right )n =(1/2 + o(1))n.\]
\item[(iii)] Suppose that $H$ has Property~A but not Property~B, and  one of the vertices $1, h$ has Property~C. Then
\[\delta_{<}(H, n)=\left (1-\frac{1}{\chi_{<}(H)}+o(1) \right )n=(1/2 + o(1))n.\]
\item[(iv)] Suppose that $H$ has Property~A but not Property~B, and neither of the vertices $1, h$ has Property~C. Then
\[\delta_{<}(H, n)=(1 - \alpha^*(H) + o(1))n.\]
\end{itemize}
\end{thm}
In all cases of Theorem~\ref{mainthm:intchro2}, except Case~(iv), the minimum degree threshold  is at least $(1/2+o(1))n$. 
Furthermore, although the degree threshold for graphs in Cases~(ii) and~(iii) are the same, they have different types of extremal examples at work.
Shortly in Section~\ref{sec:exteg}, we will show that there are three types of extremal examples. Indeed, space barriers yield Proposition~\ref{prop:spabarr} and therefore give the lower bound in Cases~(i) and~(iv), divisibility barriers provide the lower bound for Case~(ii) and  local barriers provide the lower bound  in Case~(iii) -- while we recall that for unordered graphs there are only space and divisibility barriers (see~\cite{kuhn2}).

\subsection{Notation}
Given integers $n\geq m\geq 1$, let $[m, n]:=\{m, \ldots, n\}$ and $[n]:=\{1, \ldots, n\}$.
For two subsets $X, Y$ of $[n]$, we write $X<Y$ if $x<y$ for all $x\in X$ and $y\in Y$. When $X$ consists of a single element $x$, we simply write $x<Y$.

A vertex is \textit{isolated} if it has no neighbors.
An \textit{empty graph} on $n$ vertices  consists of $n$ isolated vertices with no edges.
The empty graph on $0$ vertices is called the \textit{null graph}.
For an ordered graph $G$ and a linearly ordered set $A\subseteq V (G)$, the induced subgraph $G[A]$ is the subgraph of $G$ whose vertex set is $A$ and whose edge set consists of all of the edges of $G$ with both endpoints in $A$. We define $G\setminus X:=G[V(G)\setminus X]$.
For two disjoint subsets $A, B\subseteq V (G)$, the induced bipartite subgraph $G[A, B]$ is the subgraph of $G$ whose vertex set is $A\cup B$ and whose edge set
consists of all of the edges of $G$ with one endpoint in $A$ and the other endpoint in $B$. For convenience, we also write $G[A, A]:=G[A]$.

Given an (ordered) graph $G$, a vertex $x \in V(G)$ and a set $X\subseteq V(G)$, we define $d_G (x,X)$ to be the number of  neighbors  that $x$ has in  $X$.

For two ordered graphs $G_1$ and $G_2$ with disjoint vertex sets, the \textit{join graph}, denoted by $G_1*G_2$, is the ordered graph obtained from $G_1$ and $G_2$ by adding 
 all edges between $V(G_1)$ and $V(G_2)$, and where the vertices are ordered so that $V(G_1) < V(G_2)$ and  both $V(G_1)$ and $V(G_2)$ preserve their orders from $G_1$ and $G_2$ respectively.
Given an unordered graph $G$ and a positive integer $t$, let $G(t)$ be the graph obtained from $G$ by replacing every vertex $x\in V(G)$ by a set $V_x$ of $t$ vertices spanning an independent set, and joining $u\in V_x$ to $v\in V_y$ precisely when $xy$ is an edge in $G$. That is we replace the edges of $G$ by copies of $K_{t,t}$. We will refer to $G(t)$ as a \textit{blown-up} copy of $G$.

Throughout the paper, we
omit all floor and ceiling signs whenever these are not crucial.
The constants in the hierarchies used to state our results are chosen from right to left.
For example, if we claim that a result holds whenever $0< a\ll b\ll c\le 1$, then 
there are non-decreasing functions $f:(0,1]\to (0,1]$ and $g:(0,1]\to (0,1]$ such that the result holds
for all $0<a,b,c\le 1$  with $b\le f(c)$ and $a\le g(b)$. 
Note that $a \ll b$ implies that we may assume in the proof that, for example, $a < b$ or $a < b^2$.

\subsection{Organisation of paper}
In the next section we describe the extremal examples which show that our main result is  best possible. In Section~\ref{sec:expo} we give a high-level overview of our  approach to the regularity and absorbing methods in the ordered setting.
In Section~\ref{sec:gen}, we introduce a general framework for attacking ordered tiling problems, and show how to use it to prove Theorem~\ref{mainthm:intchro2}. In particular, in this section we state a so-called `almost perfect tiling' theorem (Theorem~\ref{thm:almtil}) and two absorbing theorems (Theorems~\ref{thm:abs} and~\ref{thm:abs2}).
In Section~\ref{sec:reg}, we formally state  Szemer{\'e}di's regularity lemma and introduce related tools. 
We then prove Theorem~\ref{thm:almtil} in Section~\ref{sec:almtil}, and prove Theorems~\ref{thm:abs} and~\ref{thm:abs2} in Section~\ref{sec:abs}.
We close the paper with some concluding remarks in Section~\ref{sec:conclu}.

\section{Extremal examples}\label{sec:exteg}
\subsection{Space barriers}
We begin this section with the proof of Proposition~\ref{prop:spabarr} which provides a general lower bound on $\delta _<(H,n)$ for \emph{all} ordered graphs $H$.

The following observation will be useful.
Suppose that $G_1$ and $H_1$ are ordered graphs and $G'_1$ and $H'_1$ are obtained from $G_1$ and $H_1$ respectively by reversing the ordering on $V(G_1)$ and $V(H_1)$. Then clearly $G_1$ contains a perfect $H_1$-tiling if and only if $G'_1$ contains a perfect $H'_1$-tiling. Further, $\alpha ^*(H_1)=\alpha ^*(H'_1)$ and if
$\alpha^*(H_1)=\frac{|H_1|-\alpha ^-
_{\ell} (H_1)+1}{\ell |H_1|}$ for some $1\leq \ell< \chi_{<}(H_1)$, then $\alpha^*(H'_1)=\frac{\alpha^+
_{\ell}(H'_1)}{\ell |H'_1|}.$ 

\medskip

\noindent\textbf{Proof of Proposition~\ref{prop:spabarr}.}
By the observation above, without loss of generality we may assume  that $\alpha^*= {\alpha_{\ell}^{+}}/({\ell\cdot h})$ for some $1\leq \ell< \chi_{<}(H)$. 
Therefore to prove the proposition, it is sufficient to prove that for every $1\leq \ell< \chi_{<}(H)$, 
there is an $n$-vertex graph $G$ with $\delta (G) \geq \left\lfloor\left(1-\frac{\alpha_{\ell}^{+}}{\ell\cdot h}\right)n\right\rfloor -1$ that does not contain a perfect $H$-tiling. 

For simplicity, we set $s:= (\alpha_{\ell}^{+} \cdot n)/h$. 
Let $A_1\cup A_2\cup\ldots\cup A_{\ell}$ be a partition of the interval $[s+1]$ such that $A_1 <A_2<\ldots <A_\ell$ and $||A_i| - |A_j||\leq 1$ for every $1\leq i, j \leq\ell$.
Define 
\begin{equation*}
G:=G_1*G_2*\ldots*G_{\ell+1},
\end{equation*}
 where $G_i$ is an empty graph defined on $A_i$ for every $1\leq i \leq \ell$, and $G_{\ell+1}$ is a complete graph defined on $[n] - \bigcup_{i=1}^{\ell}A_i$.
Note that $n-s-1\geq 0$ as $\alpha^{+}_{\ell}<\alpha^{+}_{\chi_{<}(H)}=h.$ Therefore, $G_{\ell +1}$ is well-defined, and could be a null graph (only when $h=n$). 

We claim that for every copy of $H$ in $G$, 
\begin{equation}\label{eq:intersect}
|V(H)\cap [s +1]|\leq \alpha_{\ell}^{+}. 
\end{equation}
If not, then there exists a copy of $H$ in $G$ with the vertices $v_1 < v_2 <\ldots < v_h$ such that $v_{\alpha_{\ell}^{+}+1} \in [s +1]$. 
In particular, there exists an integer $k_{\ell}\leq \ell$ such that $v_{\alpha_{\ell}^{+}+1}\in A_{k_{\ell}}$.
By the maximality of $\alpha_{\ell}^{+}$, the vertex $v_{\alpha_{\ell}^{+}+1}$ has a neighbor $v_{\ell'}$ in $H$, where $\ell'\in [\alpha_{\ell-1}^{+} + 1,\ \alpha_{\ell}^{+}]$. 
Since $A_{k_{\ell}}$ is an independent set of $G$, this implies that there exists an integer $k_{\ell-1}< k_{\ell}$ such that $v_{\alpha_{\ell-1}^{+}+1}\in A_{k_{\ell-1}}$. 
Repeat this process until we reach to $A_1$. Then we obtain an integer $1\leq\ell_0\leq \ell$ and a sequence of numbers $\ell \geq k_{\ell} >  k_{\ell-1} > \ldots > k_i > \ldots >k_{\ell_0}=1$ such that $v_{\alpha_{i}^{+}+1}\in A_{k_i}$. In particular, we have $v_{\alpha_{\ell_0}^{+}+1}\in A_{1}$. By the maximality of $\alpha_{\ell_0}^{+}$ and $\ell_0\geq 1$, the vertex $v_{\alpha_{\ell_0}^{+}+1}$ has a neighbor $v_{\ell_0'}$ in $H$, where $v_{\ell_0'} < v_{\ell_0}$. However, we run out of the space for $v_{\ell_0'}$ as $A_1$ is an independent set with the smallest vertices.

Finally, suppose that $G$ has a perfect $H$-tiling $\mathcal{H}$. Then by~(\ref{eq:intersect}) we have 
\[
\left|V(\mathcal{H})\cap[s +1]\right|
\leq \frac{n}{h}\cdot \alpha_{\ell}^{+}
=s
< s +1,
\]
which contradicts the definition of a perfect $H$-tiling.
\qed

\smallskip
We refer to such examples $G$ as \emph{space barriers} as, in this case, the obstruction to $G$ containing a perfect $H$-tiling is that the vertex class $[s+1]$ is `too big'.
\subsection{Divisibility barriers}
\begin{prop}\label{prop:divbarr}
Let $H$ be an ordered graph on $h$ vertices with $\chi_<(H)=2$. 
Suppose that $H$ has Property B.
Then for every  $n\in \mathbb N$ with $h|n$, there is an $n$-vertex ordered graph $G$ with $\delta (G) \geq\lfloor n/2\rfloor-2$ that does not contain a perfect $H$-tiling.
\end{prop}
\begin{proof}
Let $k$ be the largest integer such that $k\leq \lceil n/2\rceil$ and $k$ is not divisible by $h$. 
Let $G$ be the disjoint union of two complete graphs on vertex sets $[ k]$ and $[k+1, n]$. Note that $k\ge \lceil n/2 \rceil-1$; and the minimum degree of $G$ is $\min \{k-1, n-k-1\}\ge \lfloor n/2\rfloor -2$. 

Suppose that $G$ has a perfect $H$-tiling. Then there must be at least one copy $H'$ of $H$, for which both $[ k]\cap V(H')$ and $[k+1, n]\cap V(H')$ are non-empty. However, this is not possible for $H$ with Property~B, as there are no edges between $[k]$ and $[k+1, n]$ in $G$. 
\end{proof}
Note we call  such  graphs $G$  \textit{divisibility barriers} as the obstruction to containing a perfect $H$-tiling is a divisibility issue (in this case, the size of each of the two cliques is not divisible by $h$).
\subsection{Local barriers}
\begin{prop}\label{prop:locbarr}
Let $H$ be an ordered graph on $h$ vertices with $\chi_<(H)=2$. 
Suppose that one of the vertices $1, h$ has Property~C.
Then for every  $n\in \mathbb N$ with $h|n$, there is an $n$-vertex ordered graph $G$ with $\delta (G) = \lfloor n/2\rfloor $ that does not contain a perfect $H$-tiling.
\end{prop}
\begin{proof}
Without loss of generality, we assume that the vertex $h$ has Property~C.
Recall that $s(H)$ is the smallest vertex in $H$ that is adjacent to $h$. Since $h$ has Property~C, there exists an edge $ab$ in $H$ such that 
\begin{equation}\label{eq:locba}
s(H)\leq a < b \leq h-1.
\end{equation}

Let $G':=G_1 * G_2$, where $G_1$, $G_2$ are empty graphs on the vertex sets $[1, \lceil n/2 \rceil -1]$ and $[\lceil n/2 \rceil, n-1]$. 
Then we construct an ordered graph $G$ from $G'$ by adding the vertex $n$ and all edges between $n$ and $[\lceil n/2 \rceil, n-1]$.

Suppose that $G$ has a perfect $H$-tiling. Then there must be a copy of $H$ in $G$ such that $n$ plays the role of $h$ in it. 
By the construction of $G$, the image of $s(H)$ in $G$ lies in $[\lceil n/2 \rceil, n-1]$.
Then by~(\ref{eq:locba}), the images of $a, b$ in $G$ must lie in $[\lceil n/2 \rceil, n-1]$. 
This contradicts the fact that $[\lceil n/2 \rceil, n-1]$ is an independent set.
\end{proof}
We call such graphs $G$  \textit{local barriers} as the reason $G$ does not contain a perfect $H$-tiling is a localized issue  (in this case, there is a vertex that does not lie in a single copy of $H$).

\section{Applying the regularity and absorbing methods in the ordered setting}\label{sec:expo}
\subsection{The regularity method}
In this subsection we explain our approach to applying the regularity lemma in the ordered graph setting. Those readers unfamiliar with this result and related concepts should first read Section~\ref{sec:reg}.

Let $A_1,\dots, A_k$ be large disjoint equal size vertex classes in an (unordered) graph $G$ so that each pair $(A_i,A_j)$ (for distinct $i,j \in [k] $) is $\eps$-regular of density at least $d$, where $0<\eps<d$.
Such a structure is often found in an application of Szemer\'edi's regularity lemma and provides a framework for embedding subgraphs $H$ with $\chi (H)= k$ into $G$. Indeed, it is well-known that such a structure contains all \emph{fixed size} subgraphs $H$ of chromatic number at most $k$ (see Lemma~\ref{keylem}). In fact, for any  fixed subgraph $H$ with
$\chi (H)= k$, $G[A_1\cup\dots \cup A_k]$ must contain an almost perfect $H$-tiling. Moreover, the famous 
\emph{blow-up lemma} of Koml\'{o}s,  S\'ark\"{ozy} and Szemer\'{e}di~\cite{kssBlow} allows one to embed any almost spanning, bounded degree graph $F$ with $\chi(F)=k$ into $G[A_1\cup\dots \cup A_k]$.
These properties have been used in dozens of applications of the regularity lemma.

Ideally one would like to use such properties in the vertex ordered setting.
Similarly as before,  let $A_1,\dots, A_k$ be large disjoint equal size vertex classes in an $n$-vertex ordered graph $G$ so that each pair $(A_i,A_j)$ (for distinct $i,j \in [k]$) is $\eps$-regular of density at least $d$. 
Thus now the $A_i$s are subsets of $[n]$. 
Let $H$ be a fixed ordered graph with $\chi _< (H)=k$.
One can find a copy of $H$ in $G[A_1\cup \dots \cup A_k]$: as demonstrated by Lemma~\ref{lem:crop} in Section~\ref{sec:almtil},
one can find large subclasses $S_i\subseteq A_i$ for all $i \in [k]$ and a permutation $\sigma$ of $[k]$ so that $S_{\sigma(1)} <S_{\sigma(2)}<\ldots<S_{\sigma(k)}$. This allows us then to embed $H$ into $G[S_1\cup\dots \cup S_k]$ where the $i$th interval of $H$ is embedded into $S_{\sigma(i)}$.

However, in general it is far from true that  $G[A_1\cup \dots \cup A_k]$ should contain an almost perfect $H$-tiling. To see this consider the case when $k=2$ and $H$ is the ordered path $213$. Suppose $G[A_1,A_2]$ is in fact complete bipartite (so certainly $\eps$-regular) where $A_1<A_2$. Then each copy of $H$ in 
$G[A_1,A_2]$ must have one vertex in $A_1$ (the vertex playing the role of $1$) and two vertices in $A_2$;
so any $H$-tiling in $G[A_1,A_2]$ can only cover at most half of $A_1$.

At first sight this suggests perhaps the regularity method is not suitable for embedding large structures in ordered graphs. However, in this paper we demonstrate a method for overcoming this difficulty.
Suppose we wish to embed an (almost) perfect $H$-tiling in an ordered graph $G$ where $\chi _{<} (H)=r$. We obtain large disjoint vertex sets 
$A_1,\dots, A_k$ in $G$ so that each pair $(A_i,A_j)$ (for distinct $i,j \in [k]$) is $\eps$-regular of density at least $d$; now (i) $k$ may be significantly bigger than $r$ and (ii) the size of the classes $A_i$ may be far from equal. The class sizes and $k$ are chosen so that however the vertices from $A_1\cup\dots \cup A_k$ are labelled in $[n]$, there is a small $H$-tiling $\mathcal H$ in $G[A_1\cup \dots \cup A_k]$ such that $|A_i\cap V(\mathcal H)|/|A_j\cap V(\mathcal H)|=|A_i|/|A_j|$ for all distinct $i,j \in [k]$. As we now explain, with this property to hand, one can  easily find an almost perfect $H$-tiling in $G[A_1\cup \dots \cup A_k]$. Indeed, delete the vertices from $\mathcal H$. Still each pair $(A_i,A_j)$ is $2\eps$-regular and the ratios of the classes have been preserved. So we can  find a small $H$-tiling in $G[A_1\cup \dots \cup A_k]$ as before.
Repeating this process allows us to cover almost all the vertices in $G[A_1\cup \dots \cup A_k]$.

The challenge is to choose $k$ not too large (else $G$ will not be dense enough to guarantee such an $\eps$-regular structure $G[A_1\cup \dots \cup A_k]$) whilst ensuring the chosen ratios $|A_i|/|A_j|$ have the `ratio preservation' property described above. This latter point motivates the notion of a \emph{bottlegraph} of $H$ introduced in the next section. 

\subsection{The absorbing method}
The so-called \emph{absorbing method}, pioneered by R\"odl, Ruci\'nski and Szemer\'edi (see e.g.~\cite{rrs2}) has proved an immensely powerful technique for embedding problems in graphs and hypergraphs.
In particular, when one wishes to embed a spanning structure $F$ in a (hyper)graph $G$, the method can provide a certain `absorbing gadget' $Abs$ in $G$. With this gadget to hand, one then seeks to embed only an almost spanning subgraph $F'$ of $F$ into $G$;
$Abs$ will then have the power to extend the subgraph $F'$ into a copy of $F$ in $G$.
If the structure $F$ we seek  is a perfect $H$-tiling, then we will say $Abs$ is an \emph{$H$-absorbing set}.

The now standard approach to construct $H$-absorbing sets for perfect $H$-tilings in (hyper)graphs originates from a paper of Lo and Markstr\"om~\cite{lo}. 
Indeed, suppose one wishes to find a perfect $H$-tiling in an $n$-vertex graph $G$ where $h:=|H|$. In the simplest case,
they show that to construct an $H$-absorbing set in $G$ it suffices to show that for every pair $x,y \in V(G)$ there are $\Omega (n^{h-1})$  vertex classes $X \subseteq V(G)\setminus \{x,y\}$ of size $h-1$ so that both $X\cup \{x\} $ and $X\cup \{y\} $ span copies of $H$ in $G$.
Call such an $(h-1)$-set $X$ \emph{good for $x,y$}.

For many problems this property is relatively easy to establish. 
For example, in the (unordered) graph setting, a simple application of the regularity lemma can be used to establish this property when $G$ is an $n$-vertex graph with minimum degree more than $(1-1/\chi(H)+o(1))n$.

However, the analogous statement for ordered graphs is in general far from true. In particular, there are ordered graphs $H$ and $G$ where $\delta (G)$ is much greater than $(1-1/\chi_<(H))n$ and yet there exist pairs of vertices $x,y \in V(G)$ for which no $(h-1)$-set is good for $x,y$. For example, for any $n \in \mathbb N$ divisible by $3$, consider the complete $3$-partite ordered graph $G'$ on $[n-2]$ with classes $A_1,A_2,A_3$ of sizes $n/3-1$, $n/3-1$ and $n/3$ respectively, and where $A_1<A_2<A_3$. Obtain $G$ from $G'$ by adding vertices $n-1$ and $n$ where $n$ is adjacent to every vertex in $A_1\cup A_3$ and $n-1$ is adjacent to every vertex in $A_2 \cup A_3$.
Now choose $H$ to be the ordered graph obtained from the complete bipartite graph $K_{2,2}$ by labelling the elements in the first vertex class $1,2$; the elements in the second class $3,4$. Observe that $\chi_<(H)=2$.
Notice that $\delta (G)=2n/3-1$ and yet there are no good $3$-sets $X$ for $n-1, n$. Indeed,
this follows because any copy of $H$ containing $n$ in  $G\setminus \{n-1\}$ must use vertices in $A_1$ to play the role of $1,2$, whilst  any copy of $H$ containing $n-1$ in  $G\setminus \{n\}$ must use vertices in $A_2$ to play the role of $1,2$.

Despite this difficultly, in Theorem~\ref{thm:abs}, we are able to show the existence of an $H$-absorbing set in any $n$-vertex ordered graph $G$ with $\delta(G) >(1-1/\chi_<(H)+o(1))n$ (for every fixed ordered graph $H$).
The key is that, as made precise by Lo and Markstr\"om~\cite{lo}, to obtain an $H$-absorbing set in $G$ it is also sufficient to prove that for any $x,y \in V(G)$ there are `many' \emph{good sets} $X \subseteq V(G)$
of the same fixed (constant) size so that both $G[X\cup\{x\}]$ and $G[X\cup \{y\}]$
contain perfect $H$-tilings; see Lemma~\ref{lo} below.

The way we construct such good sets $X$ for every $x,y \in V(G)$ can be summarized by the following process -- a philosophy to absorbing that we term \emph{local-global absorbing}.
\begin{itemize}
    \item {\bf Step 1: local absorbing.} Prove that for any $x\in V(G)$, \emph{most} $y$ `close' to $x$ (with respect to the ordering on $V(G)=[n]$) are such that there are many good $(h-1)$-sets for $x,y$. 
    \item {\bf Step 2: global absorbing.} Piece together chains of the `local' good sets found in Step 1 to prove that for any pair $x,y \in V(G)$ there are many good sets $X$ of bounded size for $x,y$.
\end{itemize}
Note that our illustrative example above shows in general one cannot hope to replace the word \emph{most} in Step 1 with \emph{all}. On the other hand, the intuition why it is often easier to find good sets for $x,y \in V(G)$ where $x$ and $y$ have labels close together is that such $x$ and $y$ can often play the role of the same vertex in copies of $H$. 
The above two-step process is sufficient to prove Theorem~\ref{thm:abs}; Step 1 corresponds to Lemma~\ref{lem:connect} and Step~2 to Lemma~\ref{lem:connect2}.

On the other hand, to prove (a general case of) the other of our absorbing theorems (Theorem~\ref{thm:abs2}) we need a variant of this approach.
\begin{itemize}
    \item {\bf Step 1: local absorbing.} Prove that given any $x,y\in V(G)=[n]$ with both vertices either not too large (i.e., bounded away from $(1-o(1))n$) or not too small (i.e., not $o(n)$), there are many good $(h-1)$-sets for $x,y$. 
    \item {\bf Step 2: special global absorbing.} Prove that given any $x \in  [o(n)]$ and $y \in [n-o(n),n]$, there are many good $(h-1)$-sets for $x,y$. 
    \item {\bf Step 3: global absorbing.} Piece together chains of our  good sets found in Steps 1 and 2 to prove that for any pair $x,y \in V(G)$ there are many good sets $X$ of bounded size for $x,y$.
\end{itemize}
The intuition why in Step~2  it is often easier to find good sets for $x,y \in V(G)$ where $x$ is close to $1$ and $y$ is close to $n$ is as follows: in `most' copies of $H$ in $G$ containing $x$, $x$ must play the role of $1$, whist in `most' copies of $H$ in $G$ containing $y$, $y$ must play the role of $h$.


 \section{A general framework and the proof of Theorem~\ref{mainthm:intchro2}}\label{sec:gen}
 
\subsection{General framework}
In this section we introduce two theorems, which as well as being tools in the proof of  our main result, are applicable to the general perfect $H$-tiling problem for ordered graphs.

First, as described in the previous section, we adapt the absorbing method to the setting of ordered graphs. 
Let $H$ be an ordered graph. Given an ordered graph $G$, a set $S \subseteq V(G)$ is an \emph{$H$-absorbing set for $Q \subseteq V(G)$}, if both
$G[S]$ and $G[S\cup Q]$ contain perfect $H$-tilings. In this case we say that \emph{$Q$ is $H$-absorbed by $S$}. Sometimes we will simply refer
to a set $S \subseteq V(G)$ as an $H$-absorbing set if \emph{there exists} a non-empty set $Q \subseteq V(G)\setminus S$ that is $H$-absorbed by $S$. Roughly speaking,
the following result provides an absorbing set $Abs$ in an ordered graph $G$ of large minimum degree, where crucially $Abs$ is an $H$-absorbing set for \emph{every} not too large set of vertices $Q \subseteq V(G)\setminus Abs$.
\begin{thm}[Absorbing theorem]\label{thm:abs}
Let $H$ be an $h$-vertex ordered graph  and let $\eta >0$. 
Then there exists an $n_0\in \mathbb N$ and $0<\nu \ll \eta$ so that the following holds.
Suppose that $G$ is an $n$-vertex ordered graph  where $n\geq n_0$ and where
$$\delta (G) \geq \left (1-\frac{1}{\chi _{<} (H)}+\eta \right )n.$$
Then $V(G)=[n]$ contains a set $Abs$ so that
\begin{itemize}
\item $|Abs|\leq \nu n$;
\item $Abs$ is an $H$-absorbing set for every $W \subseteq V(G) \setminus Abs$ such that $|W| \in h \mathbb N$ and  $|W|\leq \nu ^3 n $.
\end{itemize}
\end{thm}

Theorem~\ref{thm:abs} suffices for our applications in most cases.
Indeed, it is immediately applicable to the  perfect $H$-tiling problem
for any ordered graph $H$ where the minimum degree threshold for ensuring 
a perfect $H$-tiling in an ordered graph $G$ is at least $(1-\frac{1}{\chi _{<} (H)}+o(1)  )|G|$.
In particular, we will use this theorem for Cases (i)--(iii) of Theorem~\ref{mainthm:intchro2}. 
For Case (iv) (and we suspect at least for some special cases of the general perfect $H$-tiling problem) we  require an absorbing theorem for ordered graphs with much smaller minimum degree. In this situation, some structural properties of $H$ can help us improve the absorbing argument; see Theorem~\ref{thm:abs2} below.

As indicated above, to apply the absorbing method one requires a sister  \textit{almost perfect tiling theorem}, which usually states that in a graph with large minimum degree all but $o(n)$ vertices are covered by some $H$-tiling.
Although the variety of extremal examples indicates that 
proving a sharp almost perfect $H$-tiling theorem for an arbitrary ordered graph $H$ seems to be very difficult, 
in this section we propose a \textit{general framework} for obtaining such almost perfect tiling theorems.

Let $B$ be a complete $k$-partite unordered graph with parts $U_1,\ldots, U_k$,
and $\sigma$ be a permutation of the set $[k]$. 
An \textit{interval labeling} of $B$ with respect to $\sigma$ is a bijection $\phi: V(B)\rightarrow [|B|]$ such that $\phi(U_i)<\phi(U_j)$ if $\sigma(i)<\sigma(j)$.
Given $t \in \mathbb N$, recall that $B(t)$ is a blow-up  of $B$ with vertex set $\bigcup_{x\in V(B)}V_x$, where the $V_x$s are sets of $t$ independent vertices. 
Let $(B(t), \phi)$ be the ordered graph obtained from $B(t)$ by equipping $V(B(t))$ with a vertex ordering, satisfying $V_x < V_y$ for every $x, y\in V(B)$ with $\phi(x)<\phi(y)$.
We refer to $(B(t), \phi)$ as an \textit{ordered blow-up}  of $B$. 

\begin{define}[Bottlegraph]
For an ordered graph $H$, we say that a complete $k$-partite unordered graph $B$ is a \textit{bottlegraph} assigned to $H$, if for every permutation $\sigma$ of $[k]$ and every interval labeling $\phi$ of $B$ with respect to $\sigma$, there exists a constant $t=t(B, H, \phi)$ such that the ordered blow-up $(B(t), \phi)$ contains a perfect $H$-tiling.
\end{define}

\begin{thm}[Almost perfect tiling framework]\label{thm:almtil}
Let $H$ be an ordered graph on $h$ vertices. 
Suppose that $B$ is a bottlegraph assigned to $H$.
Then for every $\eta>0$, there exists an $n_0\in \mathbb{N}$ so that every ordered graph $G$ on $n\geq n_0$ vertices with
\[
\delta(G)\geq \left(1 -\frac{1}{\chi_{cr}(B)}+\eta\right)n
\]
contains an $H$-tiling covering all but at most $\eta n$ vertices.
\end{thm}

With Theorem~\ref{thm:almtil} at hand, in order to prove that all ordered graphs with a given minimum degree contain an almost perfect $H$-tiling, it is sufficient to show that certain `interval labelled' blow-ups of a specific graph $B$ with a given critical chromatic number contain perfect $H$-tilings. The latter statement usually can be verified easily by observation or by solving a linear optimization problem.
Thus, for the general perfect $H$-tiling problem,
the heart of the problem is to choose an ordered graph $B$ whose critical chromatic number is not too big (so that the corresponding minimum degree condition in Theorem~\ref{thm:almtil} is not too high) whilst ensuring $B$ is indeed a bottlegraph assigned to $H$.
These competing forces mean it is far from immediate what the correct choice of $B$ is given an arbitrary ordered graph $H$.
However, for a given class of ordered graphs $H$ (as we will see in the case when $\chi _< (H)=2$), there might be some intuitive ways to construct a `fairly good' or even optimal bottlegraph using their structural properties.

The proofs of both Theorems~\ref{thm:abs} and~\ref{thm:almtil} rely on Szemer\'edi's regularity
lemma, which will be  formally introduced in Section~\ref{sec:reg}.
We then prove Theorem~\ref{thm:almtil} in Section~\ref{sec:almtil}, and Theorem~\ref{thm:abs} in Section~\ref{sec:abs}.
 
\subsection{Graphs with interval chromatic number 2}\label{sec:promain}
In this section, we illustrate how to apply our general framework to prove Theorem~\ref{mainthm:intchro2}.  
The following key lemma gives a construction of the bottlegraph for graphs with interval chromatic number 2.
\begin{lemma}\label{lem:bottle}
Let $H$ be an ordered graph on $h$ vertices with $\chi_{<}(H)=2$. 
Recall that 
\[
\alpha^*(H)=\min\left\{\frac{\alpha^{+}_1}{h},\ \frac{h -\alpha^{-}_1+1}{h}\right\}.
\]
Then there exists a bottlegraph $B$ of $H$ such that $\chi_{cr}(B)=1/\alpha^*(H)$.
\end{lemma}
\begin{proof}
Let $p:=\alpha^*h=\min\{\alpha^{+}_1,\ h - \alpha^{-}_1 + 1\}$. By the symmetry of the argument, without loss of generality we can assume that $p=\alpha^{+}_1$. 

We first assume that $H$ does not have Property~A. Then by the definition of $p$, we have $p\leq \lfloor h/2\rfloor$, and
\begin{equation}\label{sta:edge}
\text{all the edges of $H$ are between the intervals $[ p]$ and $[p+1,\ h]$. } 
\end{equation}
Let $a, r$ be the integers such that $h=ap+r$, where $a\geq 2$ and $0\leq r<p$.
Then we define $B$ to be the complete multipartite graph with classes $U_0, U_1, \ldots, U_a$, in which $|U_0|=r$, and $|U_1|=\ldots=|U_a|=p$. We will show that $B$ is a bottlegraph assigned to $H$.

Let $\phi$ be an interval labeling of $B$. If $r=0$, then by~\eqref{sta:edge}, $(B, \phi)$ immediately contains a copy of $H$ (i.e. $(B(1),\phi)$ contains a perfect $H$-tiling). So assume $r \not =0$.
If there exists $i\geq 1$ such that $\phi(U_i)<\phi(U_0)$, then again $(B,\phi)$ contains a copy of $H$. Therefore, without loss of generality, we can assume that $\phi(U_0)<\phi(U_1)<\ldots<\phi(U_a)$.

Let $c:=\mathrm{lcm}(p, r)$, the least common multiple of $p$ and $r$, and $t:=c/r$. 
Let $B':=B(t)$ and $U'_0, U'_1, \ldots ,U'_a$  be the partite sets of $B'$ (where $U'_i$ corresponds to $U_i$).
For a set $A\subseteq V(B')$ of size $h$, if $|A\cap U'_i|=p$, for all $0\leq i\leq a-1$ and $|A\cap U'_a|=r$, we say $A$ is a \textit{type I set}; if $|V(H)\cap U'_0|=0$, $|V(H)\cap U'_i|=p$, for all $1\leq i\leq a-1$ and $|V(H)\cap U'_a|=p+r$, we say $A$ is a \textit{type II set}.
Both type I and type II sets induce some complete multipartite graphs in $B'$, which contain a copy of $H$ by~\eqref{sta:edge}.
By the choice of $c$ and $t$, $V(B')$ can be partitioned into $t=c/r$ disjoint sets, where $c/p$ of them are of type I and $c/r - c/p$ of them are of type II; thus this ensures a perfect $H$-tiling in $B'$. So indeed $B$ is a bottlegraph of $H$. Moreover, it is easy to compute that $\chi_{cr}(B)=(a+1 - 1)\frac{h}{h-r}=a(h/ap)=h/p=1/\alpha^*(H)$.

Now we assume that $H$ has Property~A; then we have $\alpha^{-}_1\leq\lfloor h/2 \rfloor+1\leq \alpha^{+}_1$.
Observe that all the edges of $H$ are  between the intervals $[ \alpha^{-}_1 - 1]$ and $[\alpha^{+}_1+1,\ h]$.
 Let $r:=h-p$, and take $B:=K_{r, p}$.
 Note that $p\geq r=\max\{h-\alpha^{+}_1, \alpha^{-}_1 - 1\}$. Therefore, for any interval labeling $\phi$ of $B$, the ordered graph $(B,\ \phi)$ contains a copy of $H$; so $B$ is a bottlegraph assigned to $H$. Finally, we check that $\chi_{cr}(B)=h/p=1/\alpha^*(H)$. 
\end{proof}

Applying Theorem~\ref{thm:almtil} with Lemma~\ref{lem:bottle}, we immediately obtain a bound on the minimum degree that guarantees an almost perfect $H$-tiling for any $H$ with $\chi_<(H)=2$.
\begin{thm}\label{thm:almtil2}
Let $H$ be an ordered graph on $h$ vertices with $\chi_{<}(H)=2$. 
For every $\eta>0$, there exists an $n_0\in \mathbb{N}$ so that the following holds.
Every ordered graph $G$ on $n\geq n_0$ vertices with
\[
\delta(G)\geq \left(1 -\alpha^*(H)+\eta\right)n,
\]
contains an $H$-tiling covering all but at most $\eta n$ vertices.
\end{thm}

\noindent\textbf{Proof of Theorem~\ref{mainthm:intchro2}(i)--(iii).}
Our desired lower bounds on $\delta_<(H, n)$ follow immediately from the extremal examples in Section~\ref{sec:exteg}.
More specifically, the lower bound in (i) is given by the space barriers, i.e. Proposition~\ref{prop:spabarr}; the lower bound in (ii) is given by the divisibility barriers, i.e. Proposition~\ref{prop:divbarr}; the lower bound in (iii) is given by the local barriers, i.e. Proposition~\ref{prop:locbarr}.

For an arbitrary small constant $0<\eta<1$, let $\nu$ be defined as in Theorem~\ref{thm:abs}, and fix an additional constant $\eta'$ satisfying the following:
\begin{equation}\label{ass:para}
0<\eta ' \ll \nu \ll \eta.
\end{equation}
Let $n$ be a sufficiently large integer  divisible by $h$.

Recall that an ordered graph $H$ has Property A if and only if $\alpha^*(H)> 1/2$.
Then $\min\{\alpha^*(H), 1/2\}$ is equal to $\alpha^*(H)$ in Case~(i), and $1/2$ in Cases~(ii) and (iii).
Therefore, for the rest of the proof, it is sufficient to show that every ordered graph $G$ on $n$ vertices with
\[
\delta(G)\geq \left(1 - \min\{\alpha^*(H), 1/2\}+\eta\right)n
\]
contains a perfect $H$-tiling.

First of all, by Theorem~\ref{thm:abs}, there exists  an $H$-absorbing set $Abs$ so that
\begin{itemize}
\item $|Abs|\leq \nu n$;
\item $Abs$ is an $H$-absorbing set for any $W \subseteq V(G) \setminus Abs$ such that $|W| \in h \mathbb N$ and  $|W|\leq \nu ^3 n $.
\end{itemize}
Set $G':=G \setminus Abs$. 
Thus \eqref{ass:para} implies that $\delta(G')\geq \left(1 - \min\{\alpha^*(H), 1/2\}+\eta'\right)|G'|$. So by Theorem~\ref{thm:almtil2}, $G'$ contains an $H$-tiling $\mathcal{H}_1$ covering all but a set $W$ of vertices with $|W|\leq \eta' n \leq \nu^3 n$. By the definition of the $H$-absorbing set, $G[W\cup Abs]$ contains a perfect $H$-tiling $\mathcal{H}_2$. Then  $\mathcal{H}_1\cup \mathcal H_2$ is a perfect $H$-tiling of $G$.
\qed
\newline

The proof of Theorem~\ref{mainthm:intchro2}(iv) is similar but requires a stronger version of the absorbing theorem.
\begin{thm}\label{thm:abs2}
Let $H$ be an $h$-vertex ordered graph with $\chi_<(H)=2$.
Suppose that $H$ has Property~A but not Property~B, and neither of the vertices $1, h$ has Property~C.
Then for every $\eta >0$, there exists an $n_0\in \mathbb N$ and $\nu >0$ so that the following holds.
Suppose that $G$ is an $n$-vertex ordered graph  where $n\geq n_0$ and where
$$\delta (G) \geq  \eta n.$$
Then $V(G)$ contains a set $Abs$ so that
\begin{itemize}
\item $|Abs|\leq \nu n$;
\item $Abs$ is an $H$-absorbing set for any $W \subseteq V(G) \setminus Abs$ such that $|W| \in h \mathbb N$ and  $|W|\leq \nu ^3 n $.
\end{itemize}
\end{thm}
The proof of Theorem~\ref{thm:abs2} contains some technical arguments; we postpone it to Section~\ref{sec:abs}.
\newline

\noindent\textbf{Proof of Theorem~\ref{mainthm:intchro2}(iv).} The upper bound on $\delta_<(H, n)$ follows similarly from Theorems~\ref{thm:almtil2} and~\ref{thm:abs2}, while the lower bound is given by the space barriers, i.e. Proposition~\ref{prop:spabarr}.
\qed

\section{The regularity lemma and related tools}\label{sec:reg}
In the proof of our main results we will use Szemer\'edi's regularity
lemma~\cite{sze}.
In this section we will introduce all the information we require about this
result.
We first introduce some notation.
The \emph{density} of a bipartite graph with vertex classes~$A$ and~$B$ is
defined to be
$$d(A,B):=\frac{e(A,B)}{|A||B|}.$$
Given $\eps>0$, a graph $G$ and two disjoint sets $A, B\subset V(G)$, we say that the pair $(A, B)_G$ (or simply $(A, B)$ when the underlying graph is clear) is \emph{$\eps$-regular} if for all sets
$X \subseteq A$ and $Y \subseteq B$ with $|X|\ge \eps |A|$ and
$|Y|\ge \eps |B|$, we have $|d(A,B)-d(X,Y)|< \eps$. 
Given $d\in [0, 1)$, the pair $(A, B)_G$ is \emph{$(\eps,d)$-regular} if $G$ is $\eps$-regular, and $d(A,B)\geq d$.

We now collect together some useful properties of $\eps$-regular pairs.
\begin{prop}\label{prop:regspar}
For $0<\eps\ll d_2<d_1\leq 1$, there exists an integer $K=K(\eps, d_2, d_1)$ such that the following holds.
Let  $(A, B)_G$ be an $\eps$-regular pair of density $d_1$ in a graph $G$ where $|A|, |B|\geq K$. 
Then there exists a spanning subgraph $G'\subseteq G$ such that $(A, B)_{G'}$ is a $\sqrt{\eps}$-regular pair of density $d$, where $|d-d_2|\leq \eps$.
\end{prop}
{\noindent \bf Proof sketch.}
It suffices to consider the case when 
 $d_1-d_2\geq \eps$ (otherwise we set $G':=G$). 
 Let $G'$  be the  graph obtained from $G$ by retaining each edge with probability $p:=d_2/d_1$, independently of all other edges.
 Then $\mathbb{E}(d_{G'}(A, B))=pd_1=d_2$.
 
 Further, for every $X \subseteq A$ and $Y \subseteq B$ such that $|X|\geq \eps |A|$ and $|Y|\geq \eps |B|$, we have that
 $$
 \mathbb{E} (e_{G'} (X,Y))=\frac{d_2}{d_1} e_G (X,Y) \in \left ( (d_2-{\eps}/{d_1})|X||Y|, (d_2+{\eps}/{d_1})|X||Y|\right ).
 $$
 Noting that there are at most $2^{|A|+|B|}$ such pairs $X,Y$, we may repeatedly apply
  Chernoff's bound to ensure with high probability the conclusion of the proposition holds.
\qed

\begin{prop}\label{prop:regcomb}
For $0<\eps\ll d_2, d_1\leq 1/2$ with $|d_1- d_2|\leq \eps$, let $(A, B_1)_G$ and $(A, B_2)_G$ be $\eps$-regular pairs of density $d_1$ and $d_2$ respectively in a graph $G$ where $B_1$ and $B_2$ are disjoint. Then $(A, B_1\cup B_2)_G$ is a $(\sqrt{\eps}, \min\{d_1, d_2\})$-regular pair.
\end{prop}
\begin{proof}
Let $X \subseteq A$ and $Y \subseteq B_1\cup B_2$ where $|X|\geq \sqrt{\eps}|A|$ and $|Y|\geq \sqrt{\eps}(|B_1|+|B_2|)$. Let $Y_1:=Y\cap B_1$ and $Y_2:=Y\cap B_2$. 
If both $|Y_1|\geq \eps|B_1|$ and $|Y_2|\geq \eps |B_2|$ then it is easy to check the pair $X,Y$ satisfies the condition in the definition of a $\sqrt{\eps}$-regular pair. So without loss of generality it suffices to check the case when
 $|Y_1|\geq \eps|B_1|$ and $|Y_2|\leq \eps |B_2|$.  In this case $|Y_2|/|Y|\leq \sqrt{\eps}$.
Thus,
\[
\frac{e(X, Y)}{|X||Y|}-d(A, B_1\cup B_2)\geq \frac{(d_1 - \eps)(|Y|-|Y_2|)}{|Y|} - \max\{d_1, d_2\}
\geq (d_1 - \eps)- (d_1 - \eps)\sqrt{\eps} - (d_1+\eps)\geq -\sqrt{\eps},
\]
and 
\[
\frac{e(X, Y)}{|X||Y|}-d(A, B_1\cup B_2)\leq \frac{(d_1 + \eps)(|Y|-|Y_2|) + |Y_2|}{|Y|} - \min\{d_1, d_2\}
\leq (d_1 + \eps)+ (1-d_1 - \eps)\sqrt{\eps} - (d_1-\eps)\leq \sqrt{\eps}.
\]
This proves that the pair $X,Y$ satisfies the condition in the definition of a $\sqrt{\eps}$-regular pair.
\end{proof}
We will also make use of the following well-known property of regular pairs (see e.g., \cite[Fact 1.5]{KomlosSimonovits}).
\begin{lemma}[Slicing lemma]\label{lemma:regslic}
Let $(A, B)$ be an $\eps$-regular pair of density $d$, and for some $\alpha>\eps$, let $A'\subseteq A$, $B'\subseteq B$ with $|A'|\geq \alpha |A|$ and $|B'|\geq \alpha |B|$. Then $(A', B')$ is $(\eps', d-\eps)$-regular with $\eps':=\max\{\eps/\alpha, 2\eps\}$. \qed
\end{lemma}

We will apply the following degree form of  Szemer\'edi's regularity lemma~\cite{sze}. 
\begin{lemma}[Regularity lemma] \label{reglem} For every $\varepsilon > 0$ and $\ell_0 \in \mathbb{N}$ there exists $L_0 = L_0(\varepsilon, \ell_0)$ such that for every $d \in [0, 1]$ and for every graph $G$ on $n \geq L_0$ vertices there exists a partition $V_0, V_1, \ldots, V_\ell$ of $V(G)$ and a spanning subgraph $G'$ of $G$, such that the following conditions hold:
\begin{itemize}
\item [\rm (i)] $\ell_0 \leq \ell \leq L_0$;
\item [\rm (ii)] $d_{G'}(x) \geq d_G(x) - (d + \varepsilon) n$ for every $x \in V(G)$;
\item [\rm (iii)] the subgraph $G'[V_i]$ is empty for all $1 \leq i \leq \ell$;
\item [\rm (iv)] $|V_0| \leq \varepsilon n$;
\item [\rm (v)] $|V_1| = |V_2| = \ldots = |V_\ell|$;
\item [\rm (vi)] for all $1 \leq i < j \leq \ell$  either $(V_i, V_j)_{G'}$ 
is an $(\varepsilon, d)$-regular pair or $G'[V_i, V_j]$ is empty.
\end{itemize}
\end{lemma}
We call $V_1, \dots, V_\ell$ {\it clusters}, $V_0$ the {\it exceptional set} and the
vertices in~$V_0$ {\it exceptional vertices}. We refer to~$G'$ as the {\it pure graph}.
The {\it reduced graph~$R$ of~$G$ with parameters $\varepsilon$, $d$ and~$\ell_0$} is the graph whose 
vertices are $V_1, \dots , V_\ell$ and in which $V_i V_j$ is an edge precisely when $(V_i,V_j)_{G'}$
is $(\varepsilon,d)$-regular.

A $t$-partite graph with parts $W_1,\ldots, W_t$ is \textit{nearly balanced} if $||W_i|-|W_j||\leq 1$ for every $1\leq i, j \leq t$.
We will also make use of the following multipartite version of Lemma~\ref{reglem}.
\begin{lemma}[Multipartite regularity lemma] \label{multi} 
Given any integer $t \geq 2$, any $\varepsilon > 0$ and any $\ell_0 \in \mathbb{N}$ there exists $L_0 = L_0(\varepsilon,t, \ell_0)\in \mathbb N$ such that for every $d \in [0, 1]$ 
and for every nearly balanced $t$-partite graph $G=(W_1,\dots, W_t)$ on $n \geq L_0$ vertices,
 there exists an $\ell\in \mathbb N$,  a partition $W^0_i, W^1_i, \ldots, W^\ell _i$ of $W_i$ for each $i \in [t]$ and a spanning subgraph $G'$ of $G$, such that the following conditions hold:
\begin{itemize}
\item [\rm (i)] $\ell_0 \leq \ell \leq L_0$;
\item [\rm (ii)] $d_{G'}(x) \geq d_G(x) - (d + \varepsilon) n$ for every $x \in V(G)$;
\item [\rm (iii)] $|W_i^0| \leq \varepsilon n/t$ for every $i \in [t]$;
\item [\rm (iv)] $|W^j _i| = |W^{j'}_{i'}|$ for every $i,i' \in [t]$ and $j,j' \in [\ell]$;
\item [\rm (v)] for every $i,i' \in [t]$ and $j,j' \in [\ell]$  either $(W^j _i, W^{j'}_{i'})_{G'}$ 
is an $(\varepsilon, d)$-regular pair or $G'[W^j _i, W^{j'}_{i'}]$ is empty.
\end{itemize} 
\end{lemma}
Similarly as before, for  $i \in [t]$ and $j \in [\ell]$ we call the $W^j _i$ {\it clusters}, the $W_i^0$ the {\it exceptional sets} and the
vertices in the~$W_i^0$ {\it exceptional vertices}. We refer to~$G'$ as the {\it pure graph}.
The {\it reduced graph~$R$ of~$G$ with parameters $\varepsilon$, $d$ and~$\ell_0$} is the graph whose 
vertices are the $W^j _i$ (where  $i \in [t]$ and $j \in [\ell]$) and in which $W^j _i W^{j'}_{i'}$ is an edge precisely when $(W^j _i, W^{j'}_{i'})_{G'}$
is $(\varepsilon,d)$-regular.

The following well-known corollary of the regularity lemma shows that the reduced graph almost inherits the minimum degree of the original graph.
\begin{prop}\label{prop:regdegree}
Let $0< \varepsilon , d, k<1$,
 $G$ be an $n$-vertex graph with $\delta(G)\geq kn$ and $R$ be the reduced graph of $G$ obtained by applying the regularity lemma with parameters $\varepsilon, d$. Then  $\delta(R)\geq (k-2\eps -d)|R|.$\qed
\end{prop}

The next key lemma allows us to use the reduced graph $R$ of $G$ as a framework for embedding subgraphs into $G$.

\begin{lemma}[Key lemma \cite{KomlosSimonovits}]\label{keylem}
Suppose that $0 < \varepsilon < d$, that $q, t \in \mathbb{N}$ and that $R$ is a graph with $V(R) = \{v_1, \ldots, v_k\}$. 
We construct a graph $G$ as follows: replace every vertex $v_i \in V(R)$ with a set $V_i$ of $q$ vertices and replace each edge of $R$ with an $(\varepsilon,d)$-regular pair. For each $v_i \in V(R)$, let $U_i$ denote the set of $t$ vertices in $R(t)$ corresponding to $v_i$.
Let $H$ be a subgraph of $R(t)$ with maximum degree $\Delta$ and set $h := |H|$. Set $\delta := d - \varepsilon$ and $\varepsilon_0 := \delta^{\Delta}/(2 + \Delta)$.
If $\varepsilon \leq \varepsilon_0$ and $t-1 \leq \varepsilon_0q$ then there are at least $$(\varepsilon_0 q)^h \ \mbox{labelled copies of $H$ in $G$} $$ so that if $x \in V(H)$ lies in $U_i$ in R(t), then $x$ is embedded into $V_i$ in $G$.
\end{lemma}
Our applications of  Lemma~\ref{keylem} will take the following form: suppose within an ordered graph $G$ we have vertex classes $V_1<\ldots<V_k$ so that each pair $(V_i,V_j)_G$ is $(\varepsilon,d)$-regular. Then Lemma~\ref{keylem} tells us  $G$ contains (many) copies of any fixed size ordered graph $H$ with
 $\chi _< (H)=k$, where the $i$th vertex class of $H$ is embedded into $V_i$.

\section{Proof of Theorem~\ref{thm:almtil}}\label{sec:almtil}
We will apply the following result of Koml\'{o}s~\cite{Komlos}; this result shows that the critical chromatic number of $H$ governs the minimum degree threshold for the existence of almost perfect $H$-tilings in unordered graphs.
\begin{thm}[Koml\'os~{\cite[Theorem~8]{Komlos}}]\label{thm:kol}
Let $\mu>0$ and let $F$ be an unordered graph. Then there exists an $n_0=n_0(\mu, F)\in \mathbb{N}$ such that every graph $G$ on $n\geq n_0$ vertices with 
\[
\delta(G)\geq \left(1 - \frac{1}{\chi_{cr}(F)}\right)n
\]
contains an $F$-tiling covering all but at most $\mu n$ vertices.
\end{thm}
The next  result ensures that in any $k$ linear size disjoint vertex sets $A_1,\dots,A_k$ of an ordered graph $G$, one can find `nicely ordered' linear size  subsets $S_i$ of each $A_i$. As  we will see shortly, this property is crucial for our application of the regularity lemma in the proof of Theorem~\ref{thm:almtil}. As pointed out by a referee, it is also a special case of the `same type lemma' of B\'ar\'any and Valtr~\cite{bv}.

\begin{lemma}\label{lem:crop}
For $n\geq k\geq 2$, let $A_1, A_2, \ldots, A_k$ be nonempty disjoint
subsets of $[n]$. Then there exist sets $S_1, S_2, \ldots, S_k$, where $S_i\subseteq A_i$, and a permutation $\sigma=(\sigma(1), \sigma(2), \ldots, \sigma(k))$ of the set $[k]$, such that the following conditions hold for all $i, j\in[k]$:
\begin{itemize}
\item[(i)] $|S_i|\geq \lfloor |A_i|/k \rfloor$;
\item[(ii)] $S_i < S_j$ if $\sigma(i)<\sigma(j)$.
\end{itemize}
\end{lemma}
\begin{proof}
By removing elements if necessary we may assume that each $A_i$ contains a multiple of $k$ elements.
Given any $i \in [k]$, we refer to the $j$th smallest number in $A_i$ as the \emph{$j$th element of $A_i$}. Let $i_1 \in [k]$ be such that the $(|A_{i_1}|/k)$th element of $A_{i_1}$ is smaller than the $(|A_{j}|/k)$th element of $A_{j}$ for all $j \in [k]\setminus \{i_1\}$. Define $S_{i_1}$ to consist of the first $|A_{i_1}|/k$ elements of $A_{i_1}$.
Next define $i_2 \in [k]\setminus \{i_1\}$  such that the $(2|A_{i_2}|/k)$th element of $A_{i_2}$ is smaller than the $(2|A_{j}|/k)$th element of $A_{j}$ for all $j \in [k]\setminus \{i_1, i_2\}$. Define $S_{i_2}$ to contain the $t$th elements of $A_{i_2}$ where $t=(|A_{i_2}|/k)+1,\dots, 2|A_{i_2}|/k$.
Continuing in this way we define sets $S_{i_1}<S_{i_2}<\dots <S_{i_k}$ where each $S_{i_j}\subseteq A_{i_j}$ has size $|A_i|/k$ and $\{i_1,\dots, i_k\}=[k]$. This immediately implies the lemma.
\end{proof}

\noindent\textit{Proof of Theorem~\ref{thm:almtil}.}
We will fix additional constants satisfying the following hierarchy
\begin{equation}\label{eq:para}
0<\eps_1\ll \eps_2 \ll\eps \ll d, \mu_1, \mu_2\ll \eta, 1/|B|.
\end{equation}
Moreover, we choose an integer $\ell_0$ such that $\ell_0\geq n_0(\mu_1, B)$, where $n_0(\mu_1, B)$ is as defined in Theorem~\ref{thm:kol}.
In what follows, we assume that the order $n$ of our given ordered graph $G$ is sufficiently large for our estimates to hold. We now apply the regularity lemma (Lemma~\ref{reglem}) with parameters $\eps_1, d$ and $\ell_0$ to $G$ to obtain a reduced graph $R$, clusters $\{V_a, a\in V(R)\}$, an exceptional set $V_0$, and a spanning subgraph $G'\subseteq G$. 
Inequality (\ref{eq:para}) together with Proposition~\ref{prop:regdegree} implies that
\begin{equation}
\delta(R)\geq \left(1 - \frac{1}{\chi_{cr}(B)} + \frac{\eta}{2}\right)|R|.
\end{equation}
Since $|R|\geq \ell_0\geq n_0(\mu_1, B)$, we can apply Theorem~\ref{thm:kol} to $R$ to find a $B$-tiling $\mathcal{B}$ covering all but at most $\mu_1|R|$ vertices.
We delete all the clusters not contained in some copy of $B$ in $\mathcal{B}$ from $R$ and add all the vertices lying in these clusters to the exceptional set $V_0$. Thus, $|V_0|\leq \eps_1 n + \mu_1 n\leq 2\mu_1 n$.
 From now on, we denote by $R$ the subgraph of the reduced graph induced by all the remaining clusters. Thus $\mathcal{B}$ now is a perfect $B$-tiling of $R$.
 
Fix an arbitrary copy $B\in\mathcal{B}$ with partite sets $U_1,\ldots, U_k$, and let $A:=\bigcup_{a\in V(B)}V_a$ and $A_i:=\bigcup_{a\in U_i}V_a$. Since $B$ is a complete multipartite graph, repeatedly applying Propositions~\ref{prop:regspar} and~\ref{prop:regcomb} to $G'[A]$, we can find a spanning subgraph $G''\subseteq G'[A]$ such that for every distinct $i, j\in [k]$, $(A_i, A_j)_{G''}$ is $(\eps_2, d-\eps_2)$-regular.
The idea is that $G''$ is a blow-up of the bottlegraph $B$, where the complete bipartite graphs between vertex classes are replaced by $(\eps_2, d-\eps_2)$-regular pairs. We now show that this bottlegraph-like structure will ensure that $G''$ contains an almost perfect $H$-tiling. Then repeating this process for every $B\in\mathcal{B}$  will ensure the
desired almost perfect $H$-tiling in $G$.
 
Let $\alpha:=1/{(2k)}$. By Lemma~\ref{lem:crop}, there exist sets $S_i\subseteq A_i$, and a permutation $\sigma$ of $[k]$ such that $|S_i|\geq \alpha|A_i|$, and $S_i<S_j$ whenever $\sigma(i)<\sigma(j)$. Moreover, by the slicing lemma (Lemma~\ref{lemma:regslic}), we have that each $(S_i, S_j)_{G''}$ is $(\eps, d-\eps)$-regular. 
Now we apply the key lemma (Lemma~\ref{keylem}) on $G''[\cup S_i]$, and find a blown-up copy $B_1(t)$ of $B$, where $|B_1(t)\cap S_i|=|U_i|t$ for every $i$ and $t$ is a fixed integer given by the definition of the bottlegraph. Note that by the choice of the $S_i$, $B_1$ naturally has an interval ordering with respect to the permutation $\sigma$, and therefore $B_1(t)$ has a perfect $H$-tiling.
After that, we can delete $V(B_1(t))$ from $A$ (and therefore from each $A_i$);
crucially after this deletion, the ratio $|A_i|/|A_j|$ amongst all pairs of classes $A_i$, $A_j$  remains the same as before. Further, still for every distinct $i, j\in [k]$, $(A_i, A_j)_{G''}$ is $(2\eps_2, d-2\eps_2)$-regular.

These properties allow us to repeatedly apply this argument, thereby
 obtaining an $H$-tiling in $G[A]$ covering all but at most $\mu_2 |A|$ vertices. More precisely, suppose we have subsets $A'_i\subseteq A_i$ for all $i \in [k]$ where:
 (i) $|A'_i|\geq \mu _2 |A_i|$ for all $i \in [k]$; (ii) $|A'_i|/|A'_j|=|A_i|/|A_j|$ for all $i,j \in [k]$. Then by the slicing lemma, and as $2\eps_2 /\mu_2 \ll \sqrt{\eps_2} $ and $\eps_2 \ll d$, we have that $(A'_i,A'_j)_{G''}$ is 
 an $(\sqrt{\eps_2},d/2)$-regular pair for all distinct 
 $i,j \in [k]$. Thus, we can repeatedly apply the argument in the  paragraph above (now to the $A'_i$ rather than the $A_i$), whilst still retaining property (ii) and terminating the process when we obtain subsets $A'_i$ that no longer satisfy property (i). Notice that by (ii), as soon as (i) is no longer satisfied for \emph{some} $i \in [k]$, in fact $|A'_i|< \mu _2 |A_i|$ for \emph{all} $i \in [k]$.
 Thus, this  process will result in an $H$-tiling in $G[A]$ covering all but at most $\mu_2 |A|$ vertices.
 
 
 Finally, simply repeat this process for all copies of $B$ in $\mathcal{B}$; we obtain an $H$-tiling of $G$ covering all but at most $(2\mu_1+\mu_2)n\leq \eta n$ vertices.
\qed

\section{Proof of the absorbing theorems}\label{sec:abs}
To prove Theorems~\ref{thm:abs} and~\ref{thm:abs2}, we make use of the following, now standard, lemma.
\begin{lemma}\label{lo}
Let $h,s\in \mathbb N$ and $\xi >0$. Suppose that $H$ is an ordered hypergraph on $h$ vertices. Then there exists an $n_0 \in \mathbb N$ such that the following holds. Suppose that $G$ is an ordered hypergraph
on $n \geq n_0$ vertices so that, for any $x,y \in V(G)$, there are at least $\xi n^{sh-1}$ $(sh-1)$-sets $X \subseteq V(G)$ such that both $G[X \cup \{x\}]$ and $G[X \cup \{y\}]$ contain perfect $H$-tilings.
Then $V(G)$ contains a set $M$ so that
\begin{itemize}
\item $|M|\leq (\xi/2)^h n/4$;
\item $M$ is an $H$-absorbing set for any $W \subseteq V(G) \setminus M$ such that $|W| \in h \mathbb N$ and  $|W|\leq (\xi /2)^{2h} n/(32s^2 h^3)$. 
\end{itemize}
\end{lemma} 
Lemma~\ref{lo} was proven in the case when $G$ is unordered by 
Lo and Markstr\"om~\cite[Lemma 1.1]{lo}. However, the proof in the ordered setting is identical (so we do not provide a proof here). In particular, the original proof requires nowhere that the graphs are unordered.

\subsection{Proof of Theorem~\ref{thm:abs}}
To prove Theorem~\ref{thm:abs} we must show that the hypothesis of Lemma~\ref{lo} is satisfied. 
Define $B_{[n]}(x, z)$ as the set $[n] \cap [x-z, x+z]$.
The following lemma provides a step in that direction.

\begin{lemma}\label{lem:connect}
Let $H$ be an $h$-vertex ordered graph  and let $0< \eta \ll 1/h$. 
Then there exists an $n_0\in \mathbb N$ and $\rho, \gamma >0$ where $1/n_0 \ll \rho \ll \gamma \ll \eta $ and so that the following holds.
Suppose that $G$ is an ordered graph with vertex set $[n]$ where $n\geq n_0$ and where
$$\delta (G) \geq \left (1-\frac{1}{\chi_< (H)}+\eta \right )n.$$
Given any $x \in [n]$, there are at least $(1-\gamma)|B_{[n]}(x, \eta n/16)|$ elements $y \in [n]$ so that
\begin{itemize}
\item $y \in B_{[n]}(x, \eta n/16)$;
\item there are at least $\rho n^{h-1}$ $(h-1)$-sets $X \subseteq V(G)$ such that both $G[X \cup \{x\}]$ and $G[X \cup \{y\}]$  contain spanning copies of $H$.
\end{itemize}
\end{lemma}
\proof
Choose $0<\rho \ll 1/\ell _0  \ll \eps \ll \gamma \ll d \ll \eta \ll 1/h$ where $\ell _0 \in \mathbb N$, 
and let $n$ be sufficiently large. Let $G$ be as in the statement of the lemma.
Write $r:=\chi _<(H)$.

First of all, clearly there is a partition $W_1, \dots , W_t$ of $[n]$ where
\begin{itemize}
\item[(i)] $t:= \lfloor 8/\eta \rfloor$;
\item[(ii)] $|W_i|= \lfloor \frac{n}{t} \rfloor$ or $\lceil \frac{n}{t} \rceil$ for all $ i \in [t]$;
\item[(iii)] $W_i<W_j$ for every $1\leq i<j\leq t $;
\item[(iv)] there is some $i^* \in [t-1]$ so that $B_{[n]}(x, \eta n/16) \subseteq W_{i^*}\cup W_{i^*+1} $.
\end{itemize}
Note that (iii) implies that each of the $W_i$s is an interval in $[n]$.

\smallskip

Define $G_1:=G[W_1,W_2,\dots, W_t]$; that is we have deleted all edges within each $G[W_i]$.
Hence,
\begin{align}\label{G1}
\delta (G_1) \geq \left (1 -\frac{1}{r}  +\frac{2\eta}{3} \right )n.
\end{align}

Apply Lemma~\ref{multi} to $G_1$ with parameters $\eps, d, t, \ell _0$ to obtain a pure graph $G'_1$ and reduced graph $R_1$ of $G_1$, and a partition $W^0_i, W^1_i, \ldots, W^\ell _i$ of $W_i$ for each $i \in [t]$.
Crucially, we have defined $G_1$ so that if $W^{j_1}_{i_1}W^{j_2}_{i_2} \in E(R_1)$ then  $W^{j_1}_{i_1}<W^{j_2}_{i_2}$ or $W^{j_2}_{i_2}<W^{j_1}_{i_1}$.
Inequality (\ref{G1}) together with Proposition~\ref{prop:regdegree} imply that 
\begin{equation}\label{R1}
\delta (R_1) \geq (1-1/r+\eta /2)|R_1|.
\end{equation}

Now let $R^*_1$ be the induced subgraph of $R_1$ obtained by deleting all $W^j_{i^*+1}$.
Thus, we have deleted precisely a $(1/t)$th proportion of $V(R_1)$ to obtain $R^*_1$.
Therefore, (\ref{R1}) and (i) imply that 
\begin{equation}\label{R1*}
\delta (R^*_1) \geq (1-1/r+\eta /4)|R^*_1|.
\end{equation}
Write $N_{R^*_1}(x):= \{W^j _i \in V(R^*_1) \, : \, d_{G_1} (x, W^j _i) \geq \eta |W^j _i| /4\}$. The minimum degree condition on $G_1$ ensures that
\begin{align}\label{spec}
|N_{R^*_1}(x)|\geq (1-1/r +\eta /4)|R^*_1|.
\end{align}

Fix an arbitrary cluster $W^{j^*} _{i^*}$ for some $j^*\in[\ell]$.
Combining  (\ref{R1*}) and (\ref{spec}) ensures we can greedily choose clusters $W^{j_1}_{i_1}, \dots, W^{j_{r-1}}_{i_{r-1}}$ so that:
\begin{itemize}
\item[(a)] $W^{j_1}_{i_1}, \dots, W^{j_{r-1}}_{i_{r-1}}$  together with $W^{j^*} _{i^*}$ form a copy of $K_r$ in $R^*_1$;
\item[(b)] $W^{j_1}_{i_1}, \dots, W^{j_{r-1}}_{i_{r-1}} \in N_{R^*_1} (x)$;
\item[(c)] There is some $z^*\in \{0,\dots, r-1\}$ so that
$$W^{j_1}_{i_1}<\cdots < W^{j_{z^*}}_{i_{z^*}}< (W^{j^*} _{i^*} \cup \{x\}) < W^{j_{z^*+1}} _{i_{z^*+1}} < \dots < W^{j_{r-1}}_{i_{r-1}}.$$
\end{itemize}
In particular, (c) is ensured by the choice of $R^*_1$ and (iv).

By the slicing lemma (Lemma~\ref{lemma:regslic}) and the fact that $W^{j_k}_{i_k}\in N_{R^*_1}(x)$ for every $k \in [r-1]$, the pair $(W^{j^*} _{i^*}, N_{G_1}(x)\cap~W^{j_k}_{i_k})_{G'_1}$ is $(\eps ^{1/2}, d/2)$-regular.
By the definition of $(\eps^{1/2}, d/2)$-regularity, all but at most $r\eps^{1/2}|W^{j^*} _{i^*}|$ vertices $y \in W^{j^*} _{i^*}$ have degree at least $(d/2-\eps^{1/2})|N_{G_1}(x)\cap W^{j_k}_{i_k}|\geq d\eta |W^{j_k}_{i_k}|/12$ into $N_{G_1}(x)\cap W^{j_k}_{i_k}$ in $G_1$ for every $k \in [r-1]$. 
Fix such a vertex $y$.
Define $$W'_k:= N_{G_1}(x)\cap W^{j_k}_{i_k} \cap N_{G_1}(y)$$ for each $k \in [r-1]$, and note that $|W'_k|\geq d\eta |W^{j_k}_{i_k}|/12$.
Given any $i \not = j \in [r-1]$, Lemma~\ref{lemma:regslic} implies that each pair $(W'_i,W'_j)_{G'_1}$ and $(W'_i, W^{j^*} _{i^*})_{G'_1}$ are $(\eps ^{1/4}, d/4)$-regular.
Recalling that $\chi _< (H)=r$, property (c) above together with Lemma~\ref{keylem} implies that 
there are at least $\rho n^{h-1}$ $(h-1)$-sets $X \subseteq W^{j^*} _{i^*} \cup \bigcup W'_k$ 
such that both $G[X \cup \{x\}]$ and $G[X \cup \{y\}]$ span copies of $H$.

For each choice of the cluster $W^{j^*} _{i^*}$, there were at most $r\eps^{1/2} |W^{j^*} _{i^*}|$ `bad' selections for $y \in W^{j^*} _{i^*}$.
Since $|W^{0} _{i^*}|\leq \eps n/t$ this implies that for all but at most $(r\eps^{1/2} + \eps) |W_{i^*}|$ vertices $y \in W_{i^*}$, 
there are at least $\rho n^{h-1}$ $(h-1)$-sets $X \subseteq V(G)$ such that both $G[X \cup \{x\}]$ and $G[X \cup \{y\}]$ span copies of $H$.

One can  argue analogously  (now considering the induced subgraph $R^{**}_1$ of $R_1$ obtained by deleting all $W^j_{i^*}$) to conclude the following: for all but at most $(r\eps^{1/2} + \eps) |W_{i^*+1}|$ vertices $y \in W_{i^*+1}$, 
there are at least $\rho n^{h-1}$ $(h-1)$-sets $X \subseteq V(G)$ such that both $G[X \cup \{x\}]$ and $G[X \cup \{y\}]$ span copies of $H$.

Thus, (iv) above together with the fact that $|W_{i^*}|,|W_{i^*+1}| \leq  3|B_{[n]}(x, \eta n/16) |$ and $\eps \ll \gamma$ implies that the conclusion of the lemma holds.
\endproof

With Lemma~\ref{lem:connect} to hand, we can now prove the following result.
Note that Lemma~\ref{lem:connect2} together with Lemma~\ref{lo} immediately imply Theorem~\ref{thm:abs}.
Indeed, applying  Lemma~\ref{lem:connect2} ensures the hypothesis of Lemma~\ref{lo} holds, and then the latter result yields the desired absorbing set $Abs$.

\begin{lemma}\label{lem:connect2}
Let $H$ be an $h$-vertex ordered graph and $0<\eta \ll 1/h$. 
Then there exists an $n_0\in \mathbb N$ and $\xi >0$ where $1/n_0 \ll \xi \ll \eta \ll 1/h$  so that the following holds.
Set $s:= \lceil 32/\eta \rceil$.
Suppose that $G$ is an ordered graph with vertex set $[n]$ where $n\geq n_0$ and where
$$\delta (G) \geq \left (1-\frac{1}{\chi_< (H)}+\eta \right )n.$$
Given any $x, y \in [n]$, there are at least $\xi n^{sh-1}$ $(sh-1)$-sets $X \subseteq V(G)$ such that both 
 $G[X \cup \{x\}]$ and $G[X \cup \{y\}]$ contain perfect $H$-tilings.
\end{lemma}
\proof
Choose $\xi$ so that $0< \xi \ll \rho''\ll\rho' \ll \rho\ll \gamma \ll \eta$ where $\rho$ and $\gamma$ are as in Lemma~\ref{lem:connect}.
Let $G$ be as in the statement of the lemma.

The idea for the proof is straightforward: we first prove the result for every $x,y$ very close together except that instead of having $s= \lceil 32/\eta \rceil$ we have $s=1$ (call this Step~1). Then for $x,y$ slightly further apart, we have many choices of $z$ `in the middle' of $x$ and $y$. Then applying Step~1 to both $x,z$ and $y,z$ (and `gluing' the structures between $x$ and $z$, and $z$ and $y$ together) we conclude that the lemma holds for such $x,y$ except that now $s=2$. Repeating this process we deduce that for $x$ and $y$ of increasing distance, one can conclude that the lemma holds for such $x,y$, but at the expense of increasing $s$.
From this it is easy to deduce that the  lemma holds for all $x,y \in [n]$ with $s:= \lceil 32/\eta \rceil$.

First suppose $x,y \in [n]$ and $|x-y|\leq \eta n/16$. Then by Lemma~\ref{lem:connect} there are at least $\eta n/20$ vertices $z$ in 
$ B_{[n]}(x, \eta n/16)\cap B_{[n]}(y, \eta n/16)$ for which
\begin{itemize}
\item there are at least $\rho n^{h-1}$ $(h-1)$-sets $X \subseteq V(G)$ such that both $G[X \cup \{x\}]$ and $G[X \cup \{z\}]$ span copies of $H$;
\item there are at least $\rho n^{h-1}$ $(h-1)$-sets $Y \subseteq V(G)$ such that both $G[Y \cup \{y\}]$ and $G[Y \cup \{z\}]$ span copies of $H$.
\end{itemize}
Choose $z$, $X$ and $Y$ to be disjoint; there are at least $\eta \rho ^2 n^{2h-1}/(20(2h-1)!)-O(n^{2h-2})> \rho' n^{2h-1} $ choices for the set $S:=\{z\}\cup X \cup Y$.
Notice that each such set $S$ is chosen so that both  $G[S \cup \{x\}]$ and $G[S \cup \{y\}]$ contain perfect $H$-tilings.

Next, we assume $|x-y|\leq \eta n/9$.  There are at least $ \eta n/9- 2(\eta n/9 -\eta/16) = \eta n/72$ vertices $z$ such that
$$|x-z|, |y-z|\leq \eta n/16,$$
and for each such choice of $z$,
\begin{itemize}
\item there are at least $\rho' n^{2h-1}$ $(2h-1)$-sets $X \subseteq V(G)$ such that both $G[X \cup \{x\}]$ and $G[X \cup \{z\}]$ contain perfect $H$-tilings;
\item there are at least $\rho' n^{2h-1}$ $(2h-1)$-sets $Y \subseteq V(G)$ such that both $G[Y \cup \{y\}]$ and $G[Y \cup \{z\}]$ contain perfect $H$-tilings.
\end{itemize}
Indeed, the first bullet point is obtained by applying the conclusion of the last paragraph with $z$ playing the role of $y$; the last bullet point is obtained by applying the conclusion of the last paragraph with $z$ playing the role of $x$.
Similarly as before, choose disjoint $z, X, Y$; there are at least $\rho''n^{4h-1}$ choices for the set $S:=\{z\}\cup X\cup Y$, for which both $G[S \cup \{x\}]$ and $G[S \cup \{y\}]$ contain perfect $H$-tilings.

More generally, for any $x,y \in [n]$, by repeated iterations of the above argument we obtain some $t\leq s$ such that there are at least
$\xi^{1/2} n^{th-1}$ $(th-1)$-sets $X' \subseteq V(G)$ such that both 
 $G[X' \cup \{x\}]$ and $G[X' \cup \{y\}]$ contain perfect $H$-tilings.
For each such set $X'$ we have that $G\setminus X'$ contains more than $\rho n^{h}/2$ copies of $H$.
Add $s-t$ such disjoint copies of $H$  to obtain from $X'$ a set $X$. Then $X$ is as  desired and there are at least $\xi n^{sh-1}$ choices for $X$.
\endproof

\subsection{Proof of Theorem~\ref{thm:abs2}}
To prove Theorem~\ref{thm:abs2}, we need the following two lemmas to verify the hypothesis of Lemma~\ref{lo}.
\begin{lemma}\label{lem:conn1}
Let $H$ be an $h$-vertex ordered graph with $\chi_<(H)=2$, which satisfies the following properties:
\begin{itemize}
\item[\rm{(i)}] $1, \lceil h/2\rceil, h$ are isolated vertices;
\item[\rm{(ii)}] all edges of $H$ are between the intervals $A:=[2, \lceil h/2\rceil - 1]$ and $B:=[\lceil h/2\rceil + 1, h-1]$. 
\end{itemize}
Let $0<\eta < 1$. 
Then there exists an $n_0\in \mathbb N$ and $\xi >0$ where $1/n_0 \ll \xi \ll \eta ,1/h$  so that the following holds.
Suppose that $G$ is an ordered graph with vertex set $[n]$ where $n\geq n_0$ and where
\begin{equation}\label{eq: conn1de}
\delta (G) \geq \eta n.
\end{equation}
Given any $x, y \in [n]$, there are at least $\xi n^{h-1}$ $(h-1)$-sets $X \subseteq V(G)$ such that both 
 $G[X \cup \{x\}]$ and $G[X \cup \{y\}]$ span copies of $H$.
\end{lemma}
\begin{proof}
Let $H_0$ be a complete bipartite ordered graph with parts $S_0<L_0$, where $|S_0|=|A|+1$ and $|L_0|=|B|+1$. 
For a copy of $H_0\subseteq G$ and a vertex $v\in V(G)\setminus V(H_0)$, we say $H_0$ is \textit{good for $v$} if one of the following holds: (a) $v < S_0 <L_0$; (b) $S_0 < v < L_0$; (c) $S_0 <L_0 <v$.
By the assumption on $H$, if $H_0$ is good for $v$, then $G[V(H_0)\cup v]$ contains a spanning copy of $H$.
Therefore, it is sufficient to find $\xi n^{h-1}$ copies of $H_0$ in $G$ which are good for both $x$ and $y$.

Without loss of generality we assume $x<y$. 
Let $V_1:=\{v\in [n] \mid v<x\}$, $V_2:=\{v\in [n] \mid x<v<y\}$, and $V_3:=\{v\in [n] \mid v>y\}$.
By~(\ref{eq: conn1de}) and the pigeonhole principle, there exist $1\leq i\leq j\leq 3$ such that $e(G[V_i, V_j])\geq \eta n^2/13$. A standard application of the regularity method shows that there are at least $\xi n^{h-1}$ copies of $H_0$ in $G[V_i, V_j]$. By the construction of the $V_i$s,  each such copy of $H_0$ is good for both $x$ and $y$, and this completes the proof.
\end{proof}

Recall that $s(H)$ is the smallest vertex in $H$ that is adjacent to $h$, if $h$ is not isolated.
\begin{lemma}\label{lem:conn2}
Let $H$ be an $h$-vertex ordered graph with $\chi_<(H)=2$, which satisfies the following properties:
\begin{itemize}
\item[\rm{(i)}] $1$ and $\lceil h/2 \rceil$ are isolated vertices, while $h$ is not isolated;
\item[\rm{(ii)}] all edges of $H$ are between the intervals $[2,  \lceil h/2 \rceil -1]$ and $[\lceil h/2 \rceil +1, h]$;
\item[\rm{(iii)}] $1\leq s(H)<\lceil h/2 \rceil$ and $[s(H), h-1]$ is an independent set.
\end{itemize}
Let $0<\eta \ll 1/h$. 
Then there exists an $n_0\in \mathbb N$ and $\xi >0$ where $1/n_0 \ll \xi \ll \eta \ll 1/h$  so that the following holds.
Set $s:= 2(\lceil h/2 \rceil - s(H))$.
Suppose that $G$ is an ordered graph with vertex set $[n]$ where $n\geq n_0$ and
\begin{equation}\label{eq: conn2de}
\delta (G) \geq \eta n.
\end{equation}
For every $x, y \in [n]$, there are at least $\xi n^{sh-1}$ $(sh-1)$-sets $X \subseteq V(G)$ such that both 
 $G[X \cup \{x\}]$ and $G[X \cup \{y\}]$ contain perfect $H$-tilings.
\end{lemma}
\begin{proof}
Choose $0<\xi\ll\xi_1, \xi_2\ll \xi_3 \ll \varepsilon' \ll\varepsilon\ll d\ll d' \ll\eta \ll 1/h$, and without loss of generality we always assume $x<y$. 
Let $A:=[2, s(H)-1]$, $B:=[s(H), \lceil h/2 \rceil -1]$, and $C:=[\lceil h/2 \rceil +1, h-1]$.
We also write $a:=|A|$, $b:=|B|$ and $c:=|C|$, then $h=a+b+c+3$.
Note that $A$ and $B\cup C$ are independent sets of $H$, and every edge of $H$ is either between intervals $A$ and $C$, or between $B$ and $h$. 
\begin{claim}\label{claim:case1}
For $x, y\leq (1 - \eta/3)n$, there are at least $\xi_1 n^{h-1}$ $(h-1)$-sets $X_1 \subseteq V(G)$ such that both 
 $G[X_1 \cup \{x\}]$ and $G[X_1 \cup \{y\}]$ span copies of $H$.
\end{claim}
\begin{proof}
Let $H_1$ be a complete bipartite ordered graph with parts $S_1 < L_1$, where $|S_1|=a+b+1$ and $|L_1|=c+1$. 
For a copy of $H_1\subseteq G$ and a vertex $v\in V(G)\setminus V(H_1)$, we say $H_1$ is \textit{good for $v$} if either $v < S_1 <L_1$, or $S_1 < v < L_1$.
Note that $G[V(H_1)\cup v]$ contains a spanning copy of $H$, if $H_1$ is good for $v$. 
Now let
$V_1:=\{v\in [n]\mid v<x\}$, $V_2:=\{v\in [n] \mid x<v<y\}$, $V_3:=\{v\in [n] \mid y<v \leq (1 - \eta/3)n\}$, and $V_0=[(1 - \eta/3)n+1, n]$.
By~(\ref{eq: conn2de}) and the pigeonhole principle, there exists $i\in [3]$ such that 
\begin{align}\label{eqqq}
e(G[V_i, V_0])
\geq \frac{1}{3}\left(\frac{\eta n}{3}\right)\left(\eta n- \frac{\eta n}{3}\right)
\geq \frac{2\eta^2 n^2}{27}.
\end{align}
Notice that for any choice of $i\in [3]$, every copy of $H_1$ in $G[V_i, V_0]$ is good for both $x$ and $y$. So, as in the proof of Lemma~\ref{lem:conn1}, (\ref{eqqq}) implies that there are $\xi_1 n^{h-1}$ copies of $H_1$ in $G[V_i, V_0]$ which are good for both $x$ and $y$, as desired.
\end{proof}

\begin{claim}\label{claim:case2}
For $x\leq \eta n/3$ and $y\geq (1 - \eta/3)n$, there are at least $\xi_2 n^{bh-1}$ $(bh-1)$-sets $X_2 \subseteq V(G)$ such that both 
 $G[X_2 \cup \{x\}]$ and $G[X_2 \cup \{y\}]$ contain perfect $H$-tilings.
\end{claim}
\begin{proof}
Let $M:=[x+1, y-1]$.
By~(\ref{eq: conn2de}), we have $|N(y)\cap M|\geq \eta n - 2\eta n/3=\eta n/3$. Let $N$ be a subset of $N(y)\cap M$ of size $\eta n/6$; then
\[
e(G[N, M\setminus N])\geq \left(\frac{\eta n}{6}\right)\left(\eta n - \frac{2\eta n}{3} - \frac{\eta n}{6}\right)=\frac{\eta^2n^2}{36}.
\]
A standard application of the regularity method shows that there exists an $(\varepsilon', d')$-regular pair $(P', Q')$ in $G$, where $P'\subseteq N$; $Q'\subseteq M\setminus N$; $|P'|,|Q'|\geq 2\xi _3 n$.
By Lemma~\ref{lem:crop} and Lemma~\ref{lemma:regslic}, there exist sets $P\subseteq P'$ and $Q\subseteq Q'$ such that: $|P|,|Q|\geq \xi _3n$; $(P, Q)$ is an $(\varepsilon, d)$-regular pair in $G$;  either $P<Q$ or $Q<P$.
\\

\noindent\textbf{Case 1: $P<Q$.}\\
Let $H_2$ be a complete bipartite ordered graph with parts $S_2 < L_2$, where $|S_2|=a+b+1$ and $|L_2|=c+1$.
By Lemma~\ref{keylem}, there are at least $\xi_2 n^{h-1}$ copies of $H_2$ in $G[P, Q]$, and for every such $H_2$, we have $x<S_2<L_2<y$ and $S_2\subseteq P \subseteq N \subseteq N(y)$.
Recall that each edge of $H$ lies either between $A$ and $C$, or between $B$ and $h$.  
Therefore, $G[V(H_2)\cup \{x\}]$ contains a spanning copy of $H$, as $H_2$ is a complete bipartite graph. 
Similarly, $G[V(H_2)\cup \{y\}]$ also contains a spanning copy of $H$, as $H_2$ is a complete bipartite graph and $S_2\subseteq N(y)$.
Hence, there are at least $\xi_2 n^{h-1}$ $(h-1)$-sets $X'\subseteq V(G)$ such that both $G[X'\cup \{x\}]$ and $G[X'\cup \{y\}]$ contain perfect $H$-tilings. By adding $b-1$ additional disjoint copies of $H$ (which can be easily found in $G[P, Q]$ as $(P, Q)$ is an $(\varepsilon, d)$-regular pair in $G$), one can immediately see that Claim~\ref{claim:case2} holds in this case.
\\

\noindent\textbf{Case 2: $Q<P$.}\\
Let $F_1$ be the complete bipartite ordered graph with parts $S'_1 < L'_1$, where $|S'_1|=a+b+1$ and $|L'_1|=c+2$. 
Let $F_2$ be a complete bipartite ordered graph with parts $S'_2 < L'_2$, where $|S'_2|=a+b+2$ and $|L'_2|=c+1$. 
Note that both $F_1$ and $F_2$ contain a spanning copy of $H$.

Let $F_3$ be the complete bipartite ordered graph with parts $S'_3 < L'_3$, where $|S'_3|=a+b$ and $|L'_3|=c+2$. 
We say a copy of $F_3$ is \textit{good for $x$} if $x<S'_3 < L'_3$.
Note that $G[V(F_3)\cup\{x\}]$ contains a spanning copy of $H$, if $F_3$ is good for $x$.
Lastly, let $F_4$ be the complete bipartite ordered graph with parts $S'_4 < L'_4$, where $|S'_4|=a+1$ and $|L'_4|=b+c+1$. 
We say a copy of $F_4$ is \textit{good for $y$} if $S'_4 < L'_4<y$ and $L'_4\subseteq N(y)$. Observe that $G[V(F_4)\cup\{y\}]$ contains a spanning copy of $H$, if $F_4$ is good for $y$.

Let $H_3$ be the complete bipartite ordered graph with parts $S_3 < L_3$, where $|S_3|=b(a+b+1)-1$ and $L_3=b(c+2)$.
By Lemma~\ref{keylem}, there are at least $\xi_2 n^{bh-1}$ copies of $H_3$ in $G[P, Q]$, and for every such $H_3$, we have $x<S_3<L_3<y$ and $L_3\subseteq P \subseteq N(y)$.
Note that such $H_3$ can be decomposed into $b-1$ copies of $F_1$ and one good copy of $F_3$.
This indicates that $G[V(H_3)\cup \{x\}]$ contains a perfect $H$-tiling.
Similarly, such $H_3$ can also be decomposed into $b-1$ copies of $F_2$ and one good copy of $F_4$, which indicates that $G[V(H_3)\cup \{y\}]$ contains a perfect $H$-tiling.
\end{proof}

For every $x , y\leq (1 - \eta/3)n$ or $x\leq \eta n/3$ and $y\geq (1 - \eta/3)n$, simply adding enough disjoint copies of $H$ to the sets obtained from Claims~\ref{claim:case1} or~\ref{claim:case2} completes the proof.
For every $x\geq \eta n/3$ and $y\geq (1 - \eta/3)n$, there are at least $\eta n/3$ vertices $z$ (i.e.~the vertices in $[\eta n/3]$) such that: (i)~$\{y, z\}$ satisfies the condition of Claim~\ref{claim:case2}; (ii) $\{x, z\}$ either satisfies the condition of Claim~\ref{claim:case1} or Claim~\ref{claim:case2}. 
Applying Claims~\ref{claim:case1} and~\ref{claim:case2} on pairs $\{x, z\}$ and $\{y, z\}$ produces many disjoint copies of $X^{x, z}_1\cup X^{y, z}_2$ (or $X^{x, z}_2\cup X^{y, z}_2$), where $X^{x, z}_1$ refers to $(h-1)$-sets obtained from Claim~\ref{claim:case1} for $\{x, z\}$, and similarly for $X^{x, z}_2$ and $X^{y, z}_2$. 
Finally, adding enough extra disjoint copies of $H$ to $z\cup X^{x, z}_1\cup X^{y, z}_2$ (or $z\cup X^{x, z}_2\cup X^{y, z}_2$), 
we show that for every $x,y\in [n]$, there are at least $\xi n^{2bh-1}=\xi n^{sh-1}$ $(sh-1)$-sets $X \subseteq V(G)$ such that both 
 $G[X \cup \{x\}]$ and $G[X \cup \{y\}]$ contain perfect $H$-tilings.
\end{proof}

\noindent\textbf{Proof of Theorem~\ref{thm:abs2}.} Since $H$ has property A, it  satisfies the following conditions:
\begin{itemize}
\item all edges of $H$ are between the intervals $[1,  \lceil h/2 \rceil -1]$ and $[\lfloor h/2 \rfloor +2, h]$.
\item if $h$ is even, then the vertices $h/2, h/2 +1$ are isolated; if $h$ is odd, then the vertex $(h+1)/2$ is isolated. 
\end{itemize}
Furthermore, since $H$ does not have property B, at least one of $1$ and $h$ must be isolated in $H$.

We first assume that both $1, h$ are isolated, then by Definition~\ref{def:Pc} neither of them has Property~C. 
Note that $H$ satisfies the assumptions in Lemma~\ref{lem:conn1}. Together with Lemma~\ref{lo} (applied with $s=1$), this immediately implies Theorem~\ref{thm:abs2}.

Now without loss of generality, we assume that $1$ is isolated in $H$ but $h$ is not. 
Since $h$ does not have Property~C, by Definition~\ref{def:Pc}, $[s(H), h-1]$ is an independent set in $H$.
Similarly, $H$ satisfies the assumptions in Lemma~\ref{lem:conn2}, and this, together with Lemma~\ref{lo}, completes the proof.
\qed

\section{Concluding remarks}
\label{sec:conclu}
In this paper we have introduced a general framework for the perfect $H$-tiling problem in ordered graphs. This approach can be summarized as follows:
\begin{itemize}
    \item[Step 1:] Find a candidate extremal example; an $n$-vertex ordered graph $G$ with minimum degree $\delta (G)=\alpha n - O(1)$ without a perfect $H$-tiling.
    \item[Step 2:] Find a bottlegraph $B$ assigned to $H$ with $\alpha \geq 1-1/\chi_{cr}(B)$.
    \item[Step 3:] If $\alpha \geq 1-1/\chi_< (H)$ then Theorems~\ref{thm:abs} and~\ref{thm:almtil} now combine to yield the asymptotically exact threshold. Otherwise, one seeks an improved absorbing theorem, using structural information about $H$ (\'a la Theorem~\ref{thm:abs2}).
    
\end{itemize}
Despite introducing this framework, we suspect determining the perfect $H$-tiling threshold for an arbitrary $H$ will be challenging in the sense that there could be a range of different extremal examples and optimal bottlegraphs, depending on the precise structure of $H$.

On the other hand,
in the case when $H$ is an  $h$-vertex ordered graph and $12\in E(H)$ or $(h-1)h \in E(H)$, it is actually straightforward to deduce from Theorem~\ref{hs} the minimum degree threshold for forcing a perfect $H$-tiling.
\begin{prop}
Let $n, h \in \mathbb N$ such that $h | n$.
Suppose $H$ is an $h$-vertex ordered graph.
If $G$ is an $n$-vertex ordered graph with $\delta(G)\geq (1-1/h)n$, then $G$ contains a perfect $H$-tiling.

Moreover, suppose $12\in E(H)$ or $(h-1)h \in E(H)$. Then there are $n$-vertex ordered graphs with $\delta(G)\geq (1-1/h)n-1$ that do not contain a perfect $H$-tiling.

\end{prop}
\proof
Consider the unordered underlying graph $G'$ of any ordered $n$-vertex graph $G$ with $\delta(G)\geq (1-1/h)n$. Theorem~\ref{hs} implies that $G'$ contains a perfect $K_h$-tiling. Since any ordered copy of $K_h$ contains $H$, this ensures $G$ contains a perfect $H$-tiling.

For the moreover part, notice such $H$ satisfy $\alpha^*(H)=1/h$. The result then follows directly from Proposition~\ref{prop:spabarr}.
\endproof

In~\cite{Komlos}, Koml\'os determined the minimum degree threshold for an (unordered) graph to contain an $H$-tiling covering a given proportion of the vertices; it would be interesting to obtain an ordered analogue of this result.
\begin{question}\label{quesremark}
Let $s \in (0,1)$ and $H$ be an ordered graph. What is the minimum degree threshold that ensures an ordered graph $G$ contains an $H$-tiling covering at least an $s$th proportion of its vertices?
\end{question}

There has also been interest in Ramsey and Tur\'an properties of \emph{edge ordered graphs} (see e.g.~\cite{balko2, tardosbcc}; it would be interesting to study the perfect $H$-tiling problem in this setting also.

\smallskip

Other than tiling problems, there are many natural embedding problems to consider for ordered graphs. We now raise a couple of such problems. Here by an \emph{ordered cycle} we just mean that it is a copy of some cycle $C$ where $V(C)$ has been assigned an ordering.

\begin{question}
Let $k \geq 2$ be a fixed integer and let $s \in (0,1)$. What is the minimum degree threshold that ensures an ordered $n$-vertex graph $G$ contains a copy  of some ordered cycle $C$ of length at least $sn$ and with
$\chi_<(C)=k$?
\end{question}
The following question can be viewed as raising an ordered version of the well-known El-Zahar conjecture~\cite{el}.
\begin{question}\label{newnew}
Let $k,t \geq 2$ be  fixed integers and let $\mathcal C=\{C^1,\dots, C^t\}$ be a fixed family of $t$ (not necessarily distinct) ordered cycles where  $\chi_<(C^i)=k$ for each $i \in [t]$. Set  $n_i:=|C^i|$.
What is the minimum degree threshold that ensures an ordered graph $G$ on
$n:=n_1+\dots+n_t$ vertices contains vertex-disjoint copies of each  cycle $C^1,\dots,C^t$ that together cover the vertex set of $G$?
\end{question}
We suspect our approaches to regularity and absorbing will be useful for attacking this problem for large $n$. Variants of Question~\ref{newnew} (e.g., when $\mathcal C$ contains cycles of different interval chromatic number) would also be interesting to investigate.

{\bf Remark:} Since this paper was submitted, the third author and Freschi~\cite{vo2} have asymptotically determined $\delta _<(H,n)$ for all ordered graphs $H$ with $\chi_<(H)\geq 3$. Their approach both relies on tools from this paper, as well as introducing new ideas. They have also given an asymptotic solution to Question~\ref{quesremark}. Further, Hurley, Joos and Lang~\cite{hjl} have proven some very general tiling results, including a generalisation of the K\"uhn--Osthus tiling theorem that allows tiles to be different and to have size that grows with the size of the host graph. As pointed out in~\cite{hjl}, it would be interesting to seek analogous results in the setting of ordered graphs too.

\section{Acknowledgements}
The authors are grateful to the BRIDGE strategic alliance between the University of Birmingham
and the University of Illinois at Urbana-Champaign. In particular, much of the research in this paper
was carried out whilst the third author was visiting UIUC. The authors are also grateful to the referees for their helpful and careful reviews.

\end{document}